\documentclass[11pt]{amsart}
\usepackage{amssymb, latexsym, mathrsfs, color, tikz, multirow}
\usepackage{graphics}
\input{xy}
\xyoption{all}

\usepackage[colorlinks=true, pdfstartview=FitV,
 linkcolor=blue,citecolor=blue,urlcolor=blue]{hyperref}

\numberwithin{equation}{section}
\theoremstyle{plain}
\newtheorem{Thm}[equation]{Theorem}

\newtheorem{Cor}[equation]{Corollary}
\newtheorem{Lem}[equation]{Lemma}
\newtheorem{Conj}[equation]{Conjecture}
\newtheorem {prob}{Problem}

\theoremstyle{definition}
\newtheorem{Def}[equation]{Definition}
\newtheorem{Exa}[equation]{Example}
\newtheorem{Rmk}[equation]{Remark}

\begin{document}

\title [A correspondence between rigid modules  and simple curves]
{A correspondence between rigid modules over path algebras and simple curves on Riemann surfaces}

\author[K.-H. Lee]{Kyu-Hwan Lee$^{\star}$}
\thanks{$^{\star}$This work was partially supported by a grant from the Simons Foundation (\#318706).}
\address{Department of
Mathematics, University of Connecticut, Storrs, CT 06269, U.S.A.}
\email{khlee@math.uconn.edu}

\author[K. Lee]{Kyungyong Lee$^{\dagger}$}
\thanks{$^{\dagger}$This work was partially supported by the University of Nebraska--Lincoln, Korea Institute for Advanced Study, AMS Centennial Fellowship, and NSA grant H98230-16-1-0059.}
\address{Department of Mathematics, University of Nebraska--Lincoln, Lincoln, NE 68588, U.S.A.,
and Korea Institute for Advanced Study, Seoul 02455, Republic of Korea}
\email{klee24@unl.edu; klee1@kias.re.kr}

\begin{abstract}
We propose a conjectural correspondence between the set of rigid indecomposable modules over the path algebras of acyclic quivers and the set of certain non-self-intersecting curves on Riemann surfaces, and prove the correspondence for the $2$-complete rank $3$ quivers.

\end{abstract}

\maketitle

\section{Introduction}

In the study of  the category of modules over a ring, geometric objects have often been used  to describe the structures. In particular, the following problem  has been considered fundamental. (For a small fraction of references, see \cite{AB, MSW, BZ11, ZZZ, CS, AI}.)

\begin{prob}
Let $R$ be a ring. Find a function $f$ from a (sub)set of $R$-modules to a set of geometric objects so that the size of the (asymptotic) ext group between two modules $M$ and $N$ can be measured by the intersections of $f(M)$ and $f(N)$.   
\end{prob}

The Homological Mirror Symmetry (HMS), proposed by Kontsevich  \cite{Kon}, is one of the phenomena which answer this problem. The existence of such symmetry implies that there is a symplectic manifold $S$ such that  the number of intersections between two Lagrangians on $S$ is closely related to the dimension of the ext group between the corresponding modules.

Pursuing this direction, in this paper, we restrict ourselves to the following problem.

\begin{prob} \label{prob-2}
Let $R$ be an hereditary algebra. Find a function $f$ from the set of indecomposable $R$-modules to a set of geometric objects so that the non-vanishing of the self-extension group of an indecomposable module $M$ is precisely detected by the existence of the self-intersection of $f(M)$.   
\end{prob}

Every finite-dimensional hereditary algebra over an algebraically closed field is Morita equivalent to the path algebra of an acyclic quiver, i.e., a quiver without oriented cycles (See, e.g., \cite{ASiSk}).  The number of vertices of a quiver is referred to as the rank of the quiver. The dimension vectors of indecomposable modules over a path algebra are called (positive) roots.  A root $\alpha$ is real if the Euler inner product $\langle\alpha,\alpha\rangle$ is equal to 1, and imaginary if $\langle\alpha,\alpha\rangle\leq 0$.

We first consider the case that  $R$ is the path algebra of a 2-complete quiver (that is, an acyclic quiver with at least two arrows between every pair of vertices), and define a bijective function
 $$f:\{\text{indecomposable  modules corresponding to positive real roots}\}\longrightarrow \{\text{admissible curves}\},   
$$
where admissible curves are certain paths on a Riemann surface (see Definition \ref{def-adm}). Then we formulate the following conjecture:
\begin{Conj} \label{conj-int}  
For an indecomposable $R$-module $M$, we have  $Ext^1(M,M)=0$   if and only if $f(M)$ has no self-intersections.  
\end{Conj}

In this paper, we prove this conjecture  for 2-complete rank 3 quivers. 
When $Ext^1(M,M)=0$, the module $M$ is called rigid, and the dimension vector of a rigid  indecomposable module is called a real Schur root. To explain our result, we let 
$$
\mathcal{Z}:=\{(a,b,c)\in\mathbb{Z}^3\ :  \ \text{gcd}(|b|,|c|)=1\}.
$$
For each $z=(a,b,c)\in \mathcal{Z}$, define a curve $\eta_z$ on the universal cover of a triangulated torus, consisting of two symmetric spirals and a line segment, so that $a$ determines the number of times the spirals revolve and $(b,c)$ determines the slope of the line segment. See Examples \ref{ex-main} (2). The curves $\eta_z$, $z \in \mathcal Z$, have no self-intersections. Now our result (Theorem \ref{thm-main}) is the following:

\begin{Thm}
Let $R$ be the path algebra of a 2-complete rank 3 quiver. Then 
there is a natural bijection between the set of rigid indecomposable  modules and the sef of curves $\eta_z$, $z \in \mathcal Z$.
\end{Thm}

This shows that real Schur roots  are very special ones among all real roots in general.   
 Our proof  is achieved by expressing each  real Schur root in terms of a sequence of simple reflections that corresponds to a non-self-intersecting path.\footnote{After the first version of this paper was posted on the arXiv, Felikson and Tumarkin \cite{FT} proved Conjecture \ref{conj-int}  for all 
2-complete quivers. Moreover they characterized $c$-vectors in the same seed, using a collection of pairwise non-crossing admissible curves satisfying a certain word property.} See Example \ref{exa-sub}.

For the general case, let $R$ be the path alagebra of any acyclic quiver. We still define an onto function
$$
g: \{\text{admissible curves}\} \longrightarrow \{\text{indecomposable modules corresponding to positive real roots}\},
$$ and propose the following (See Conjecture \ref{main_conj}):

\begin{Conj}
For an indecomposable $R$-module $M$, we have $Ext^1(M,M)=0$ if and only if $g^{-1}(M)$ contains a non-self-crossing curve.
\end{Conj}

  As tests for known cases, we prove this conjecture for equioriented quivers of types $A$ and $D$, and for $A_2^{(1)}$ and all rank $2$ quivers. We also consider the highest root of a quiver of type $E_8$ and provide such a path. 
 If this conjecture holds true, then it gives an elementary geometric (and less recursive) criterion to distinguish real Schur roots among all positive real roots. 
 
 There have been a number of known criteria to tell whether a given real root is a real Schur root, some of which are in terms of subrepresentations (due to Schofield \cite{S}), braid group actions (due to Crawley-Boevey \cite{C-B}),  cluster variables (due to Caldero and Keller \cite{CK}), or $c$-vectors (due to Chavez \cite{Ch}).  Building on a result of Igusa--Schiffler \cite{IS} and Baumeister--Dyer--Stump--Wegener \cite{BDSW}, Hubery and Krause  \cite{HK} characterized real Schur roots in terms of non-crossing partitions. There are also combinatorial descriptions for $c$-vectors in the same seed due to Speyer--Thomas \cite{ST} and Seven \cite{Se}. However none of these is of geometric nature, and most of them rely on heavy recursive procedures which are hard to apply in practice.

Also a better description for real Schur roots is still needed to  help understand  a base step of the non-commutative HMS for path algebras. Note that the recent work of Shende--Treumann--Williams--Zaslow \cite{STWZ,STW,TWZ} suggests HMS for certain (not-necessarily acyclic) quivers including the ones coming from bicolored graphs on surfaces.

Our conjecture suggests  the existence of the HMS phenomenon for the path algebra over an arbitrary acyclic quiver. In a subsequent project, we plan to investigate  the HMS for path algebras over various quivers.

\subsection*{Acknowledgments} We thank Cheol-Hyun Cho, Christof Geiss, Ralf Schiffler, Hugh Thomas, Pavel Tumarkin,  Jerzy Weyman, and Nathan Williams  for helpful discussions. We also thank an anonymous referee for letting us know of \cite{C-B}.  K.-H. L. gratefully acknowledges support from the Simons Center for Geometry and Physics at which some of the research for this paper was performed.

\vskip 1cm

\section{A Conjectural Correspondence} \label{Schur}

\subsection{The statement of conjecture}\label{conj_subsection} Let $\mathcal{Q}$ be an acyclic (connected) quiver with $N$ vertices labeled by $I:=\{1,...,N\}$. Denote by $S_N$ the permutation group on $I$. 
Let $P_{\mathcal{Q}}\subset S_N$ be the set of all permutations $\sigma$ such that 
there is no arrow  from $\sigma(j)$ to $\sigma(i)$  for any $j>i$ on $\mathcal Q$.
Note that if there exists an oriented path passing through all $N$ vertices on $\mathcal Q$, in particular, if there is at least one arrow between every pair of vertices, then $P_{\mathcal{Q}}$ consists of a unique permutation. 

For each $\sigma\in P_{\mathcal{Q}}$, we define a labeled Riemann surface  $\Sigma_{\sigma}$\footnote{The punctured discs appeared in Bessis' work \cite{Be}.  For better visualization, here we prefer to use an alternative description using compact Riemann surfaces with one or two marked points.} as follows.
Let $G_1$ and $G_2$ be two identical copies of a regular $N$-gon. Label the edges of each of the two  $N$-gons
 by $T_{\sigma(1)}, T_{\sigma(2)}, \dots , T_{\sigma(N)}$ counter-clockwise. On $G_i$, let $L_i$ be the line segment from the center of $G_i$ to the common endpoint of  $T_{\sigma(N)}$ and $T_{\sigma(1)}$.    Fix the orientation of every edge of $G_1$ (resp.  $G_2$) to be 
 counter-clockwise (resp. clockwise) as in the following picture. 
 
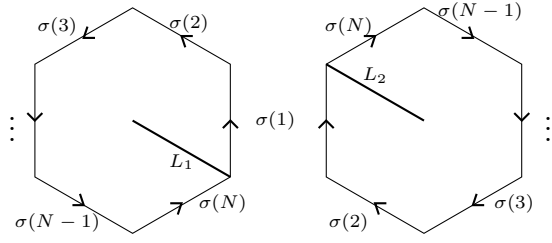
\begin{center}
 \begin{tikzpicture}[scale=0.5]
\node at (2.4,-2.2){\tiny{$\sigma(N)$}};
\node at (1.5,2.6){\tiny{$\sigma(2)$}};
\node at (3.8,0){\tiny{$\sigma(1)$}};
\node at (-2.0,-2.7){\tiny{$\sigma(N-1)$}};
\node at (-2.0,2.5){\tiny{$\sigma(3)$}};
\node at (-3.2,0){\vdots};
\draw (0,0) +(30:3cm) -- +(90:3cm) -- +(150:3cm) -- +(210:3cm) --
+(270:3cm) -- +(330:3cm) -- cycle;
\draw [thick] (2.4,-0.2) -- (2.6,0)--(2.8,-0.2);
\draw [thick] (1.4,1.95) -- (1.3,2.25)--(1.6,2.25);  
\draw [thick] (-1.0,2.2) -- (-1.3,2.2)--(-1.2,2.5);  
\draw [thick] (-2.4,0.2) -- (-2.6,0)--(-2.8,0.2);   
\draw [thick] (1.0,-2.2) -- (1.3,-2.2)--(1.2,-2.5);  
\draw [thick] (-1.4,-1.95) -- (-1.3,-2.25)--(-1.6,-2.25);  
\draw [thick] (0,0)--(2.6,-1.5);  
\node at (1.3,-1.1){\tiny{$L_1$}};
\end{tikzpicture}
\begin{tikzpicture}[scale=0.5]
\node at (-1.3,1.2){\tiny{$L_2$}};
\draw [thick] (0,0)--(-2.6,1.5);  
\node at (2.4,-2.2){\tiny{$\sigma(3)$}};
\node at (1.5,2.9){\tiny{$\sigma(N-1)$}};
\node at (3.3,0){\vdots};
\node at (-2.0,-2.7){\tiny{$\sigma(2)$}};
\node at (-2.0,2.5){\tiny{$\sigma(N)$}};
\draw (0,0) +(30:3cm) -- +(90:3cm) -- +(150:3cm) -- +(210:3cm) --
+(270:3cm) -- +(330:3cm) -- cycle;
\draw [thick] (-2.4,-0.2) -- (-2.6,0)--(-2.8,-0.2);
\draw [thick] (-1.4,1.95) -- (-1.3,2.25)--(-1.6,2.25);  
\draw [thick] (1.0,2.2) -- (1.3,2.2)--(1.2,2.5);  
\draw [thick] (2.4,0.2) -- (2.6,0)--(2.8,0.2);   
\draw [thick] (-1.0,-2.2) -- (-1.3,-2.2)--(-1.2,-2.5);  
\draw [thick] (1.4,-1.95) -- (1.3,-2.25)--(1.6,-2.25);  
\end{tikzpicture}
 \end{center}

 Let $\Sigma_{\sigma}$ be the (compact) Riemann surface of genus $\lfloor \frac{N-1}{2}\rfloor$
obtained by gluing together the two $N$-gons with all the edges of the same label identified according 
to their orientations.  The edges of the $N$-gons become $N$ different curves in $\Sigma_\sigma$. If $N$ is odd, all the vertices of the two $N$-gons 
are identified to become one point in $\Sigma_\sigma$ and the curves obtained from the edges are loops. If $N$ is even, two distinct
 vertices are shared by all curves. Let $\mathcal{T}$ be the set of all curves, i.e., $\mathcal{T}={T}_1\cup\cdots{T}_N\subset \Sigma_\sigma$, and $V$ be the set of the vertex (or vertices) on $\mathcal{T}$.  

Consider the Cartan matrix  corresponding  to  $\mathcal{Q}$  and the root system of the associated Kac--Moody algebra. The simple root corresponding to $i \in I$ will be denoted by $\alpha_i$.    The simple reflections of the Weyl group $W$ will be denoted by $s_i$, $i \in I$.  Let $\mathfrak W$ be the set of words $\mathfrak w=i_1i_2 \cdots i_k$ from the alphabet $I$ such that no two consecutive letters $i_p$ and $i_{p+1}$ are the same.   For each element $w\in W$, let $R_w\subset \mathfrak W$ be the set of words $i_1i_2 \cdots i_k$ such that $w=s_{i_1}s_{i_2} \cdots s_{i_k}$. Recall that the set of positive real roots and the set of reflections in $W$ are in one-to-one correspondence. Also note that if there are at least two arrows between every pair of vertices on $\mathcal{Q}$, then $R_w$ contains a unique element.  Define $$
R:=\bigcup_{\tiny{\begin{array}{c}w\in W\\w:\text{reflection}\end{array}}}R_w\subset \mathfrak W.
$$  

\begin{Def} \label{def-adm} Let $\sigma\in P_{\mathcal{Q}}$.
A $\sigma$-\emph{admissible} curve is a continuous function $\eta:[0,1]\longrightarrow \Sigma_{\sigma}$ such that

1) $\eta(x)\in V$ if and only if  $x\in\{0,1\}$;

2) there exists $\epsilon>0$ such that $\eta([0,\epsilon])\subset L_1$ and $\eta([1-\epsilon,1])\subset L_2$;

3) if $\eta(x)\in \mathcal{T}\setminus V$ then $\eta([x-\epsilon,x+\epsilon])$ meets $\mathcal{T}$ transversally for sufficiently small $\epsilon>0$;

4) and   $\upsilon(\eta)\in R$, where $\upsilon(\eta):={i_1}\cdots {i_k}$ is given by 
$$\{x\in(0,1) \ : \ \eta(x)\in \mathcal{T}\}=\{x_1<\cdots<x_k\}\quad \text{ and }\quad \eta(x_\ell)\in T_{i_\ell}\text{ for }\ell\in\{1,...,k\}.$$ 

If $\sigma$ is clear from the context, a $\sigma$-admissible curve will be just called an admissible curve.
\end{Def}

Note that for every $\mathfrak w\in R$, there is a $\sigma$-admissible curve $\eta$ 	with $\upsilon(\eta)=\mathfrak{w}$. In particular, every positive real root can be represented by some admissible curve(s).

\begin{Exa} \label{ex-main} 
Let $N=3$, and $\mathcal{Q}$ be the rank 3 acyclic quiver with double arrows between every pair of vertices as follows. 
 
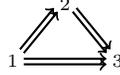
\begin{center}
 \begin{tikzpicture}[scale=0.5]
 \node at (0,0){\tiny{$1$}};
\node at (2.8,0){\tiny{$3$}};
\node at (1.4,1.5){\tiny{$2$}};
\draw [->, thick] (1.6,1.3)--(2.5,0.2);
\draw [->, thick] (1.7,1.4)--(2.6,0.3);
\draw [->, thick] (0.2,0.3)--(1.1,1.4);
\draw [->, thick] (0.3,0.2)--(1.2,1.3);
\draw [->, thick] (0.3,0.05)--(2.5,0.05);
\draw [->, thick] (0.3,-0.1)--(2.5,-0.1);
\end{tikzpicture}
\end{center}
Let $\sigma\in S_3$ be the trivial permutation $id$. We have $P_{\mathcal{Q}}=\{id\}$. 

\noindent (1) First we consider a positive real root $\alpha_1+6\alpha_2+2\alpha_3=s_2s_3\alpha_1$ and its corresponding reflection $w=s_2s_3s_1s_3s_2$. Then $R_w=\{23132\}\subset R$, and the following red curve becomes a $\sigma$-admissible curve $\eta$ on $\Sigma_\sigma$, with $\upsilon(\eta)=23132$. The picture on the right shows several copies of $\eta$ on the universal cover of $\Sigma_\sigma$, where each horizontal line segment represents $T_1$, vertical $T_3$, and diagonal $T_2$. Clearly $\eta$ has a self-intersection.

 
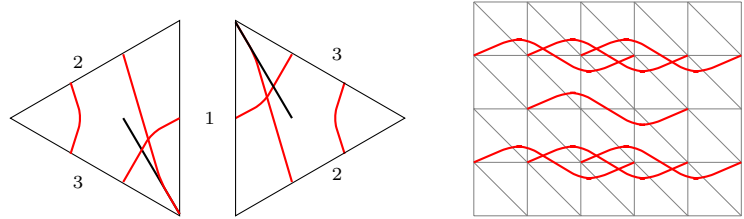
\begin{center}
 \begin{tikzpicture}[scale=0.5]
 \node at (2.3,0){\tiny{$1$}};
\node at (-1.2,-1.7){\tiny{$3$}};
\node at (-1.2,1.5){\tiny{$2$}};
\draw (0,0) +(60:3cm) -- +(180:3cm) -- +(300:3cm) -- cycle;
\draw [thick] (0,0) -- (1.5,-2.55); 
\draw [red, thick] (0,1.7)
[rounded corners] --(1,-1.7) -- (1.5,-2.55); 
\draw [red, thick] (-1.4,0.95)
[rounded corners] --(-1.1,0) -- (-1.4,-0.95); 
\draw [red, thick] (0,-1.7)
[rounded corners] --(0.75,-0.4) -- (1.5,0); 
\end{tikzpicture}
\begin{tikzpicture}[scale=0.5]
\draw [red, thick] (0,-1.7)
[rounded corners] --(-1,1.7) -- (-1.5,2.55); 
\draw [red, thick] (1.4,-0.95)
[rounded corners] --(1.1,0) -- (1.4,0.95); 
\draw [red, thick] (0,1.7)
[rounded corners] --(-0.75,0.4) -- (-1.5,0); 
\draw [thick] (0,0) -- (-1.5,2.55); 
\node at (1.2,1.7){\tiny{$3$}};
\node at (1.2,-1.5){\tiny{$2$}};
\draw (0,0) +(0:3cm) -- +(120:3cm) -- +(240:3cm) -- cycle;
\end{tikzpicture}
\qquad
\begin{tikzpicture}[scale=0.25mm]
\draw [help lines] (0,0) grid (5,4);
\draw [help lines] (0,1)--(1,0);
\draw [help lines] (0,2)--(2,0);
\draw [help lines] (0,3)--(3,0);
\draw [help lines] (0,4)--(4,0);
\draw [help lines] (1,4)--(5,0);
\draw [help lines] (2,4)--(5,1);
\draw [help lines] (3,4)--(5,2);
\draw [help lines] (4,4)--(5,3);
\draw [red, thick] (2-1,2) 
[rounded corners] -- (2.7-1,2.3) -- (3-1,2.3) -- (3.5-1,2) -- (4-1,1.7) -- (4.3-1,1.7)--(5-1,2);
\draw [red, thick] (2,2-1) 
[rounded corners] -- (2.7,2.3-1) -- (3,2.3-1) -- (3.5,2-1) -- (4,1.7-1) -- (4.3,1.7-1)--(5,2-1);
\draw [red, thick] (2-1,2-1) 
[rounded corners] -- (2.7-1,2.3-1) -- (3-1,2.3-1) -- (3.5-1,2-1) -- (4-1,1.7-1) -- (4.3-1,1.7-1)--(5-1,2-1);
\draw [red, thick] (2-2,2-1) 
[rounded corners] -- (2.7-2,2.3-1) -- (3-2,2.3-1) -- (3.5-2,2-1) -- (4-2,1.7-1) -- (4.3-2,1.7-1)--(5-2,2-1);
\draw [red, thick] (2,2+1) 
[rounded corners] -- (2.7,2.3+1) -- (3,2.3+1) -- (3.5,2+1) -- (4,1.7+1) -- (4.3,1.7+1)--(5,2+1);
\draw [red, thick] (2-1,2+1) 
[rounded corners] -- (2.7-1,2.3+1) -- (3-1,2.3+1) -- (3.5-1,2+1) -- (4-1,1.7+1) -- (4.3-1,1.7+1)--(5-1,2+1);
\draw [red, thick] (2-2,2+1) 
[rounded corners] -- (2.7-2,2.3+1) -- (3-2,2.3+1) -- (3.5-2,2+1) -- (4-2,1.7+1) -- (4.3-2,1.7+1)--(5-2,2+1);
\end{tikzpicture}
 \end{center}
 
\noindent (2) Next we consider another positive real root
\begin{equation}\label{root-example} 1662490\alpha_1+4352663\alpha_2+11395212\alpha_3=(s_3s_2s_1)^4s_2s_3s_2s_1s_2s_3\alpha_2,\end{equation} and its corresponding reflection $w=(s_3s_2s_1)^4s_2s_3s_2s_1s_2s_3s_2s_3s_2s_1s_2s_3s_2(s_1s_2s_3)^4$.  Below we draw a copy of a $\sigma$-admissible curve $\eta$ with $\upsilon(\eta)=(321)^42321232321232(123)^4$ on the universal cover of $\Sigma_\sigma$. Here $\eta$ has no self-intersection.
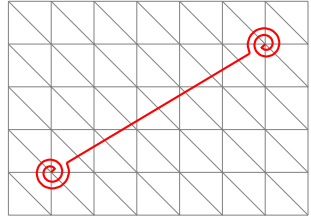
\begin{center}
\begin{tikzpicture}[scale=0.2mm]
\draw [help lines] (0,0) grid (7,5);
\draw [help lines] (0,1)--(1,0);
\draw [help lines] (0,2)--(2,0);
\draw [help lines] (0,3)--(3,0);
\draw [help lines] (0,4)--(4,0);
\draw [help lines] (0,5)--(5,0);
\draw [help lines] (1,5)--(6,0);
\draw [help lines] (2,5)--(7,0);
\draw [help lines] (3,5)--(7,1);
\draw [help lines] (4,5)--(7,2);
\draw [help lines] (5,5)--(7,3);
\draw [help lines] (6,5)--(7,4);
\draw [thick,red] (1,1)--(1.1,1.1);  
\draw [thick,red] (6,4)--(5.9,3.9);  
\draw [thick,red] (1.36,1.22)--(6-0.36,4-0.22);  
\draw[thick,red,scale=0.5,domain=0.75:13.12,smooth,variable=\t]
plot ({(0.2+0.05*\t)*cos(\t r)+2},{(0.2+0.05*\t)*sin(\t r)+2});
\draw[thick,red,scale=0.5,domain=3.89:16.26,smooth,variable=\t]
plot ({(0.03+0.05*\t)*cos(\t r)+12},{(0.03+0.05*\t)*sin(\t r)+8});
\end{tikzpicture}
\end{center}

\end{Exa}

\begin{Conj}\label{main_conj}
Let $\Gamma_{\sigma}$ be the set of (isotopy classes of) $\sigma$-admissible curves $\eta$ such that  $\eta$ has no     self-intersections, i.e., $\eta(x_1)=\eta(x_2)$ implies $x_1=x_2$ or $\{x_1,x_2\}=\{0,1\}$. For each $\eta\in \Gamma_{\sigma}$, let $w\in W$ be the reflection such that $\upsilon(\eta)\in R_w$, and let $\beta(\eta)$ be the positive real root corresponding to $w$. Then $\{ \beta(\eta)\ : \ \eta\in \cup_{\sigma\in P_{\mathcal{Q}}}\Gamma_{\sigma}\}$ is precisely the set of real Schur roots for $\mathcal{Q}$.
\end{Conj}

\begin{Rmk}
The correspondence $\eta \mapsto \beta (\eta)$ is not one-to-one in general.  If there are at least two arrows between every pair of vertices on $\mathcal{Q}$, then the conjecture predicts that it would be a bijection.
\end{Rmk}

\subsection{Type $A$ Quivers}
In this subsection we prove Conjecture~\ref{main_conj} for equioriented quivers of type $A$. Let $\mathcal{Q}$ be the following quiver:
$$
1\longrightarrow2\longrightarrow\cdots\longrightarrow n
$$ Since all positive real roots are Schur, it is enough to show that every positive real root can be realized as $\beta(\eta)$ for some $\eta\in \Gamma_{id}$. Each positive real root is equal to $s_{i}s_{i+1}\cdots s_{j-1}\alpha_{j}$ for some $i\leq j\in\{1,...,n\}$, and the corresponding reflection is $w=s_{i}\cdots s_{j-1}s_{j}s_{j-1}\cdots s_i$. There exists an admissible curve $\eta$ with no self-intersections and  $\upsilon(\eta)=i\cdots (j-1)j(j-1)\cdots i\in R_w$  as depicted in the following picture, so we are done.
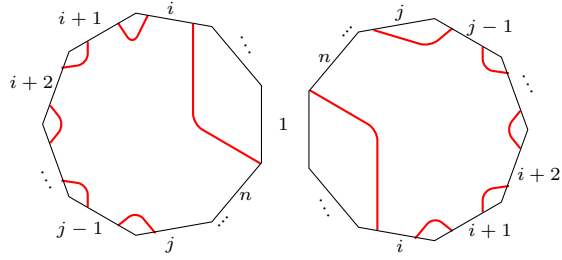
\begin{center}

\begin{tikzpicture}[scale=0.5]
\draw [red, thick] (0,-2.9) 
[rounded corners] -- (-0.5,-2.3)--(-1,-2.7);
\draw [red, thick] (2.8,-1.05)
[rounded corners] -- (1,0) -- (1,2.67);  
\draw [red, thick] (-1.8,2.2) 
[rounded corners] -- (-1.8,1.6)--(-2.5,1.5);
\draw [red, thick] (-1.0,2.7) 
[rounded corners] -- (-0.6,2.1)--(-0.2,2.9);
\draw [red, thick] (-2.8,0.5) 
[rounded corners] -- (-2.4,0)--(-2.8,-0.5);
\draw [red, thick] (-1.8,-2.2) 
[rounded corners] -- (-1.8,-1.6)--(-2.5,-1.5);
\node at (2.4,-1.9){\tiny{$n$}};
\node at (3.4,0){\tiny{$1$}};
\node at (0.4,-3.2){\tiny{$j$}};
\node at (0.4,3.1){\tiny{$i$}};
\node at (-2.0,-2.8){\tiny{$j-1$}};
\node at (-2.0,2.8){\tiny{$i+1$}};
\node at (-3.3,1.1){\tiny{$i+2$}};
\node at (-2.8,-1.7){\tiny{$\cdot$}};
\node at (-2.9,-1.5){\tiny{$\cdot$}};
\node at (-3.0,-1.3){\tiny{$\cdot$}};
\node at (2.3,2.3){\tiny{$\cdot$}};
\node at (2.4,2.1){\tiny{$\cdot$}};
\node at (2.5,1.9){\tiny{$\cdot$}};
\node at (1.7,-2.7){\tiny{$\cdot$}};
\node at (1.8,-2.6){\tiny{$\cdot$}};
\node at (1.9,-2.5){\tiny{$\cdot$}};
\draw (0,0) +(20:3cm) -- +(60:3cm) -- +(100:3cm) -- +(140:3cm) --
+(180:3cm) -- +(220:3cm) -- +(260:3cm) -- +(300:3cm) -- +(340:3cm) -- cycle;
\end{tikzpicture}
\begin{tikzpicture}[scale=0.5]
\node at (-2.4,1.9){\tiny{$n$}};
\node at (-2.3,-2.3){\tiny{$\cdot$}};
\node at (-2.4,-2.1){\tiny{$\cdot$}};
\node at (-2.5,-1.9){\tiny{$\cdot$}};
\node at (-1.7,2.7){\tiny{$\cdot$}};
\node at (-1.8,2.6){\tiny{$\cdot$}};
\node at (-1.9,2.5){\tiny{$\cdot$}};
\draw [red, thick] (0,-2.9) 
[rounded corners] -- (0.5,-2.3)--(1,-2.7);
\draw [red, thick] (1.8,2.2) 
[rounded corners] -- (1.8,1.6)--(2.5,1.5);
\draw [red, thick] (-1.1,2.65) 
[rounded corners] -- (0.4,2.2)--(1.0,2.7);
\draw [red, thick] (2.8,0.5) 
[rounded corners] -- (2.4,0)--(2.8,-0.5);
\draw [red, thick] (1.8,-2.2) 
[rounded corners] -- (1.8,-1.6)--(2.5,-1.5);
\draw [red, thick] (-2.8,1.05)
[rounded corners] -- (-1,0) -- (-1,-2.67);  
\draw (0,0) +(0:3cm) -- +(40:3cm) -- +(80:3cm) -- +(120:3cm) --
+(160:3cm) -- +(200:3cm) -- +(240:3cm) -- +(280:3cm) -- +(320:3cm) -- cycle;
\node at (-0.4,-3.1){\tiny{$i$}};
\node at (-0.4,3.1){\tiny{$j$}};
\node at (2.0,-2.7){\tiny{$i+1$}};
\node at (2.0,2.7){\tiny{$j-1$}};
\node at (3.3,-1.2){\tiny{$i+2$}};
\node at (3.0,1.2){\tiny{$\cdot$}};
\node at (3.1,1.0){\tiny{$\cdot$}};
\node at (2.9,1.4){\tiny{$\cdot$}};
\end{tikzpicture}

\end{center}

\subsection{Type $D$ Quivers}
In this subsection we prove Conjecture~\ref{main_conj} for the following quiver:
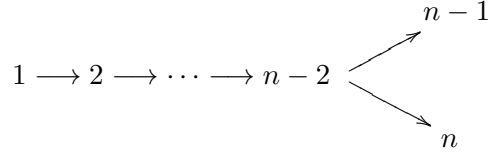
\[ \xymatrix@R-=0.3 cm{ & & n-1  \\ 1\longrightarrow 2\longrightarrow \cdots \longrightarrow  n-2 \hspace{- 1 cm} &  \ar[ru] \ar[rd] & \\  & & n \phantom{1} } \]
For the corresponding root system of type $D_n$,  all positive real roots are Schur.  Each positive real root is equal to one of the following:
\begin{align}
& s_{i}s_{i+1}\cdots s_{j-1}\alpha_{j}, & & 1 \le i\leq j \le n-1 , \label{root-f} \\ 
& s_{i}s_{i+1}\cdots s_{n-2}\alpha_{n},& & 1 \le i \le n-1 , \label{root-s} \\
&(s_{j}s_{j+1}\cdots s_{n-2})(s_{n}s_{n-1}s_{n-2}\cdots s_{i+1})\alpha_{i},& & 1 \le i < j \le n-1. \label{root-t}
\end{align}
The roots in \eqref{root-f} are of the same form as in type $A$. 
For each of the corresponding reflections to the roots in \eqref{root-s}, there exists an admissible curve $\eta$ with no self-intersections on $\Sigma_\sigma$, where $\sigma$ is either the permutation interchanging only $n-1$ and $n$  (see the picture below) or the trivial permutation $id$.  
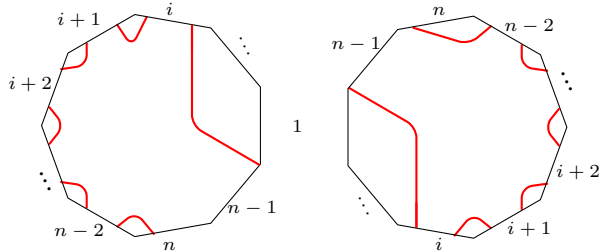
\begin{center}
\begin{tikzpicture}[scale=0.5]
\draw [red, thick] (0,-2.9) 
[rounded corners] -- (-0.5,-2.3)--(-1,-2.7);
\draw [red, thick] (1,2.0) -- (1,2.67);
\draw [red, thick] (-1.8,2.2) 
[rounded corners] -- (-1.8,1.6)--(-2.5,1.5);
\draw [red, thick] (-1.0,2.7) 
[rounded corners] -- (-0.6,2.1)--(-0.2,2.9);
\draw [red, thick] (-2.8,0.5) 
[rounded corners] -- (-2.4,0)--(-2.8,-0.5);
\draw [red, thick] (-1.8,-2.2) 
[rounded corners] -- (-1.8,-1.6)--(-2.5,-1.5);
\node at (3.8,0){\tiny{$1$}};
\node at (0.4,-3.2){\tiny{$n$}};
\node at (2.6,-2.2){\tiny{$n-1$}};
\node at (0.4,3.1){\tiny{$i$}};
\node at (-2.0,-2.8){\tiny{$n-2$}};
\node at (-2.0,2.8){\tiny{$i+1$}};
\node at (-3.3,1.1){\tiny{$i+2$}};
\node at (-2.8,-1.7){$\cdot$};
\node at (-2.9,-1.5){$\cdot$};
\node at (-3.0,-1.3){$\cdot$};
\draw (0,0) +(20:3cm) -- +(60:3cm) -- +(100:3cm) -- +(140:3cm) --
+(180:3cm) -- +(220:3cm) -- +(260:3cm) -- +(300:3cm) -- +(340:3cm) -- cycle;
\draw [red, thick] (2.8,-1.05)
[rounded corners] -- (1,0) -- (1,2.67);  
\node at (2.3,2.3){\tiny{$\cdot$}};
\node at (2.4,2.1){\tiny{$\cdot$}};
\node at (2.5,1.9){\tiny{$\cdot$}};
\end{tikzpicture}
\begin{tikzpicture}[scale=0.5]
\node at (-2.3,-2.3){\tiny{$\cdot$}};
\node at (-2.4,-2.1){\tiny{$\cdot$}};
\node at (-2.5,-1.9){\tiny{$\cdot$}};
\draw [red, thick] (-2.8,1.05)
[rounded corners] -- (-1,0) -- (-1,-2.67);  
\draw [red, thick] (0,-2.9) 
[rounded corners] -- (0.5,-2.3)--(1,-2.7);
\draw [red, thick] (1.8,2.2) 
[rounded corners] -- (1.8,1.6)--(2.5,1.5);
\draw [red, thick] (-1.1,2.65) 
[rounded corners] -- (0.4,2.2)--(1.0,2.7);
\draw [red, thick] (2.8,0.5) 
[rounded corners] -- (2.4,0)--(2.8,-0.5);
\draw [red, thick] (1.8,-2.2) 
[rounded corners] -- (1.8,-1.6)--(2.5,-1.5);
\draw [red, thick] (-1,-2.0) -- (-1,-2.67);
\draw (0,0) +(0:3cm) -- +(40:3cm) -- +(80:3cm) -- +(120:3cm) --
+(160:3cm) -- +(200:3cm) -- +(240:3cm) -- +(280:3cm) -- +(320:3cm) -- cycle;
\node at (-0.4,-3.1){\tiny{$i$}};
\node at (-2.6,2.2){\tiny{$n-1$}};
\node at (-0.4,3.1){\tiny{$n$}};
\node at (2.0,-2.7){\tiny{$i+1$}};
\node at (2.0,2.7){\tiny{$n-2$}};
\node at (3.3,-1.2){\tiny{$i+2$}};
\node at (3.0,1.2){$\cdot$};
\node at (3.1,1.0){$\cdot$};
\node at (2.9,1.4){$\cdot$};
\end{tikzpicture}
\end{center}
For each of the corresponding reflections to the roots in \eqref{root-t}, there exists an admissible curve $\eta$ with no self-intersections on $\Sigma_{id}$. Such curve is given below, where we omit drawing the edges $1,...,i-1$ with which $\eta$ does not intersect.
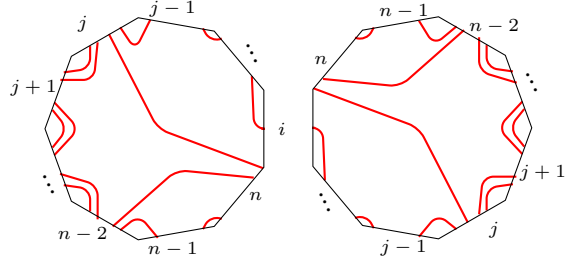
\begin{center}
\begin{tikzpicture}[scale=0.5]
\draw [red, thick] (0,-2.9) 
[rounded corners] -- (-0.5,-2.3)--(-1,-2.7);
\draw [red, thick] (1.2,2.67)
[rounded corners] -- (1.4,2.3)--(1.7,2.4);   
\draw [red, thick] (-1.8,2.2) 
[rounded corners] -- (-1.8,1.6)--(-2.5,1.5);
\draw [red, thick] (-1.0,2.7) 
[rounded corners] -- (-0.6,2.1)--(-0.2,2.9);
\draw [red, thick] (-1.8,-2.2) 
[rounded corners] -- (-1.8,-1.6)--(-2.5,-1.5);
\draw [red, thick] (2.5,1.4) 
[rounded corners] -- (2.56,0)--(2.76,0);  
\draw [red, thick] (-1.3,2.52) 
[rounded corners] --(0,0)--(2.8,-1.05);   
\draw [red, thick] (-1.6,2.3) 
[rounded corners] -- (-1.7,1.3)--(-2.59,1.3);   
\draw [red, thick] (-2.75,0.7) 
[rounded corners] -- (-2.1,0)--(-2.75,-0.7);
\draw [red, thick] (-2.8,0.5) 
[rounded corners] -- (-2.4,0)--(-2.8,-0.5);
\draw [red, thick] (-1.6,-2.41) 
[rounded corners] -- (-1.6,-1.5)--(-2.6,-1.2); 
\draw [red, thick] (2.55,-1.3) 
[rounded corners] -- (0.5,-1.1)--(-1.2,-2.6);  
\draw [red, thick] (1.2,-2.67)
[rounded corners] -- (1.4,-2.3)--(1.7,-2.4);   
\node at (0.4,-3.2){\tiny{$n-1$}};
\node at (2.6,-1.8){\tiny{$n$}};
\node at (2.6,1.8){$\cdot$};
\node at (2.5,2){$\cdot$};
\node at (2.4,2.2){$\cdot$};
\node at (3.3,0){\tiny{$i$}};
\node at (0.4,3.2){\tiny{$j-1$}};
\node at (-2.0,-2.8){\tiny{$n-2$}};
\node at (-2.0,2.8){\tiny{$j$}};
\node at (-3.3,1.1){\tiny{$j+1$}};
\node at (-2.8,-1.7){$\cdot$};
\node at (-2.9,-1.5){$\cdot$};
\node at (-3.0,-1.3){$\cdot$};
\draw (0,0) +(20:3cm) -- +(60:3cm) -- +(100:3cm) -- +(140:3cm) --
+(180:3cm) -- +(220:3cm) -- +(260:3cm) -- +(300:3cm) -- +(340:3cm) -- cycle;
\end{tikzpicture}
\begin{tikzpicture}[scale=0.5]
\node at (-2.6,1.8){\tiny{$n$}};
\node at (-2.6,-1.8){$\cdot$};
\node at (-2.5,-2){$\cdot$};
\node at (-2.4,-2.2){$\cdot$};
\draw [red, thick] (0,-2.9) 
[rounded corners] -- (0.5,-2.3)--(1,-2.7);
\draw [red, thick] (1.8,2.2) 
[rounded corners] -- (1.8,1.6)--(2.5,1.5);
\draw [red, thick] (0,2.85) 
[rounded corners] -- (0.4,2.2)--(1.0,2.7);
\draw [red, thick] (2.8,0.5) 
[rounded corners] -- (2.4,0)--(2.8,-0.5);
\draw [red, thick] (1.8,-2.2) 
[rounded corners] -- (1.8,-1.6)--(2.5,-1.5);
\draw [red, thick] (-1.2,-2.67)
[rounded corners] -- (-1.4,-2.3)--(-1.7,-2.4);   
\draw [red, thick] (-2.5,-1.4) 
[rounded corners] -- (-2.56,0)--(-2.76,0);  
\draw [red, thick] (1.3,-2.52) 
[rounded corners] --(0,0)--(-2.8,1.05);   
\draw [red, thick] (1.6,-2.3) 
[rounded corners] -- (1.7,-1.3)--(2.59,-1.3);   
\draw [red, thick] (2.75,-0.7) 
[rounded corners] -- (2.1,0)--(2.75,0.7);
\draw [red, thick] (1.6,2.41) 
[rounded corners] -- (1.6,1.5)--(2.6,1.2); 
\draw [red, thick] (-2.55,1.3) 
[rounded corners] -- (-0.5,1.1)--(1.2,2.6);  
\draw [red, thick] (-1.2,2.67)
[rounded corners] -- (-1.4,2.3)--(-1.7,2.4);   
\draw (0,0) +(0:3cm) -- +(40:3cm) -- +(80:3cm) -- +(120:3cm) --
+(160:3cm) -- +(200:3cm) -- +(240:3cm) -- +(280:3cm) -- +(320:3cm) -- cycle;
\node at (-0.4,-3.2){\tiny{$j-1$}};
\node at (-0.4,3.1){\tiny{$n-1$}};
\node at (2.0,-2.7){\tiny{$j$}};
\node at (2.0,2.7){\tiny{$n-2$}};
\node at (3.3,-1.2){\tiny{$j+1$}};
\node at (3.0,1.2){$\cdot$};
\node at (3.1,1.0){$\cdot$};
\node at (2.9,1.4){$\cdot$};
\end{tikzpicture}
\end{center}
\subsection{A quiver of type $E_8$}
In this subsection we consider the following quiver:
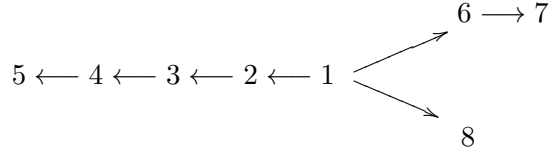
\[ \xymatrix@R-=0.3 cm{ & & \quad 6\longrightarrow 7  \\ 5\longleftarrow 4\longleftarrow 3\longleftarrow 2\longleftarrow  1 \hspace{- 1 cm} &  \ar[ru] \ar[rd] & \\  & & \quad 8 \phantom{\longrightarrow 1} } \]
Again every positive real root is Schur.  The highest positive real root $$6\alpha_1+5\alpha_2+4\alpha_3+3\alpha_4+2\alpha_5+4\alpha_6+2\alpha_7+3\alpha_8$$ can be given by   $(s_8s_7\cdots s_2s_1)^5(s_8s_7\cdots s_2)\alpha_1$, and one of non-reduced expressions for the corresponding reflection is $$(s_8s_7\cdots s_2s_1)^5(s_8s_7\cdots s_2)s_1(s_2\cdots s_7s_8)(s_1s_2\cdots s_7s_8)^5,$$ which gives rise to the following non-self-intersecting curve on $\Sigma_{id}$. 

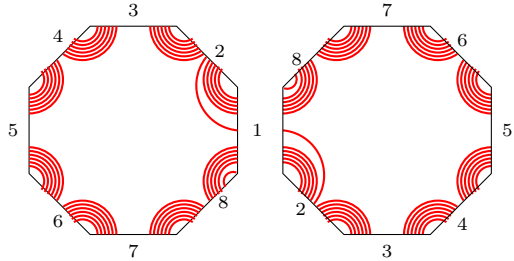
\begin{center}
\begin{tikzpicture}[scale=0.5]
\node at (0,-3.2){\tiny{$7$}};
\node at (2.4,-1.9){\tiny{$8$}};
\node at (2.3,2.1){\tiny{$2$}};
\node at (3.3,0){\tiny{$1$}};
\node at (0,3.2){\tiny{$3$}};
\node at (-2.0,-2.4){\tiny{$6$}};
\node at (-2.0,2.5){\tiny{$4$}};
\node at (-3.2,0){\tiny{$5$}};
\draw (0,0) +(22.5:3cm) -- +(67.5:3cm) -- +(112.5:3cm) -- +(157.5:3cm) --
+(202.5:3cm) -- +(247.5:3cm) -- +(292.5:3cm) -- +(337.5:3cm) -- cycle;
\draw [red, thick] (0.95,2.77) arc (180:305:0.4cm);
\draw [red, thick] (0.85,2.77) arc (180:305:0.5cm);
\draw [red, thick] (0.75,2.77) arc (180:305:0.6cm);
\draw [red, thick] (0.65,2.77) arc (180:305:0.7cm);
\draw [red, thick] (0.55,2.77) arc (180:305:0.8cm);
\draw [red, thick] (0.45,2.77) arc (180:305:0.9cm);
\draw [red, thick] (-0.95,2.77) arc (360:235:0.4cm);
\draw [red, thick] (-0.85,2.77) arc (360:235:0.5cm);
\draw [red, thick] (-0.75,2.77) arc (360:235:0.6cm);
\draw [red, thick] (-0.65,2.77) arc (360:235:0.7cm);
\draw [red, thick] (-0.55,2.77) arc (360:235:0.8cm);
\draw [red, thick] (-0.45,2.77) arc (360:235:0.9cm);
\draw [red, thick] (0.95,-2.77) arc (180:55:0.4cm);
\draw [red, thick] (0.85,-2.77) arc (180:55:0.5cm);
\draw [red, thick] (0.75,-2.77) arc (180:55:0.6cm);
\draw [red, thick] (0.65,-2.77) arc (180:55:0.7cm);
\draw [red, thick] (0.55,-2.77) arc (180:55:0.8cm);
\draw [red, thick] (0.45,-2.77) arc (180:55:0.9cm);
\draw [red, thick] (-0.95,-2.77) arc (0:125:0.4cm);
\draw [red, thick] (-0.85,-2.77) arc (0:125:0.5cm);
\draw [red, thick] (-0.75,-2.77) arc (0:125:0.6cm);
\draw [red, thick] (-0.65,-2.77) arc (0:125:0.7cm);
\draw [red, thick] (-0.55,-2.77) arc (0:125:0.8cm);
\draw [red, thick] (-0.45,-2.77) arc (0:125:0.9cm);
\draw [red, thick] (-2.77,0.95) arc (270:395:0.4cm);
\draw [red, thick] (-2.77,0.85) arc (270:395:0.5cm);
\draw [red, thick] (-2.77,0.75) arc (270:395:0.6cm);
\draw [red, thick] (-2.77,0.65) arc (270:395:0.7cm);
\draw [red, thick] (-2.77,0.55) arc (270:395:0.8cm);
\draw [red, thick] (-2.77,0.45) arc (270:395:0.9cm);
\draw [red, thick] (-2.77,-0.95) arc (90:-35:0.4cm);
\draw [red, thick] (-2.77,-0.85) arc (90:-35:0.5cm);
\draw [red, thick] (-2.77,-0.75) arc (90:-35:0.6cm);
\draw [red, thick] (-2.77,-0.65) arc (90:-35:0.7cm);
\draw [red, thick] (-2.77,-0.55) arc (90:-35:0.8cm);
\draw [red, thick] (-2.77,-0.45) arc (90:-35:0.9cm);
\draw [red, thick] (2.45,-1.5) arc (215:70:0.25cm);    
\draw [red, thick] (2.77,-0.85) arc (90:215:0.5cm);
\draw [red, thick] (2.77,-0.75) arc (90:215:0.6cm);
\draw [red, thick] (2.77,-0.65) arc (90:215:0.7cm);
\draw [red, thick] (2.77,-0.55) arc (90:215:0.8cm);
\draw [red, thick] (2.77,-0.45) arc (90:215:0.9cm);
\draw [red, thick] (2.77,0.85) arc (270:145:0.5cm);
\draw [red, thick] (2.77,0.75) arc (270:145:0.6cm);
\draw [red, thick] (2.77,0.65) arc (270:145:0.7cm);
\draw [red, thick] (2.77,0.55) arc (270:145:0.8cm);
\draw [red, thick] (2.77,0.45) arc (270:145:0.9cm);
\draw [red, thick] (2.77,0) arc (265:141:1.2cm);
\end{tikzpicture}
\begin{tikzpicture}[scale=0.5]
\draw [red, thick] (-2.77,0) arc (85:-39:1.2cm);
\draw [red, thick] (-2.45,1.5) arc (35:-110:0.25cm);   
\draw [red, thick] (-2.77,0.85) arc (270:395:0.5cm);
\draw [red, thick] (-2.77,0.75) arc (270:395:0.6cm);
\draw [red, thick] (-2.77,0.65) arc (270:395:0.7cm);
\draw [red, thick] (-2.77,0.55) arc (270:395:0.8cm);
\draw [red, thick] (-2.77,0.45) arc (270:395:0.9cm);
\draw [red, thick] (-2.77,-0.85) arc (90:-35:0.5cm);
\draw [red, thick] (-2.77,-0.75) arc (90:-35:0.6cm);
\draw [red, thick] (-2.77,-0.65) arc (90:-35:0.7cm);
\draw [red, thick] (-2.77,-0.55) arc (90:-35:0.8cm);
\draw [red, thick] (-2.77,-0.45) arc (90:-35:0.9cm);
\draw [red, thick] (2.77,-0.95) arc (90:215:0.4cm);
\draw [red, thick] (2.77,-0.85) arc (90:215:0.5cm);
\draw [red, thick] (2.77,-0.75) arc (90:215:0.6cm);
\draw [red, thick] (2.77,-0.65) arc (90:215:0.7cm);
\draw [red, thick] (2.77,-0.55) arc (90:215:0.8cm);
\draw [red, thick] (2.77,-0.45) arc (90:215:0.9cm);
\draw [red, thick] (2.77,0.95) arc (270:145:0.4cm);
\draw [red, thick] (2.77,0.85) arc (270:145:0.5cm);
\draw [red, thick] (2.77,0.75) arc (270:145:0.6cm);
\draw [red, thick] (2.77,0.65) arc (270:145:0.7cm);
\draw [red, thick] (2.77,0.55) arc (270:145:0.8cm);
\draw [red, thick] (2.77,0.45) arc (270:145:0.9cm);
\draw [red, thick] (0.95,2.77) arc (180:305:0.4cm);
\draw [red, thick] (0.85,2.77) arc (180:305:0.5cm);
\draw [red, thick] (0.75,2.77) arc (180:305:0.6cm);
\draw [red, thick] (0.65,2.77) arc (180:305:0.7cm);
\draw [red, thick] (0.55,2.77) arc (180:305:0.8cm);
\draw [red, thick] (0.45,2.77) arc (180:305:0.9cm);
\draw [red, thick] (-0.95,2.77) arc (360:235:0.4cm);
\draw [red, thick] (-0.85,2.77) arc (360:235:0.5cm);
\draw [red, thick] (-0.75,2.77) arc (360:235:0.6cm);
\draw [red, thick] (-0.65,2.77) arc (360:235:0.7cm);
\draw [red, thick] (-0.55,2.77) arc (360:235:0.8cm);
\draw [red, thick] (-0.45,2.77) arc (360:235:0.9cm);
\draw [red, thick] (0.95,-2.77) arc (180:55:0.4cm);
\draw [red, thick] (0.85,-2.77) arc (180:55:0.5cm);
\draw [red, thick] (0.75,-2.77) arc (180:55:0.6cm);
\draw [red, thick] (0.65,-2.77) arc (180:55:0.7cm);
\draw [red, thick] (0.55,-2.77) arc (180:55:0.8cm);
\draw [red, thick] (0.45,-2.77) arc (180:55:0.9cm);
\draw [red, thick] (-0.95,-2.77) arc (0:125:0.4cm);
\draw [red, thick] (-0.85,-2.77) arc (0:125:0.5cm);
\draw [red, thick] (-0.75,-2.77) arc (0:125:0.6cm);
\draw [red, thick] (-0.65,-2.77) arc (0:125:0.7cm);
\draw [red, thick] (-0.55,-2.77) arc (0:125:0.8cm);
\draw [red, thick] (-0.45,-2.77) arc (0:125:0.9cm);
\node at (0,3.2){\tiny{$7$}};
\node at (-2.4,1.9){\tiny{$8$}};
\node at (-2.3,-2.1){\tiny{$2$}};
\node at (0,-3.2){\tiny{$3$}};
\node at (2.0,2.4){\tiny{$6$}};
\node at (2.0,-2.5){\tiny{$4$}};
\node at (3.2,0){\tiny{$5$}};
\draw (0,0) +(22.5:3cm) -- +(67.5:3cm) -- +(112.5:3cm) -- +(157.5:3cm) --
+(202.5:3cm) -- +(247.5:3cm) -- +(292.5:3cm) -- +(337.5:3cm) -- cycle;
\end{tikzpicture}
\end{center}
On the actual Riemann surface $\Sigma_{id}$, this is just a spiral around one vertex followed by another spiral around the other vertex.

\subsection{Rank 2 Quivers}
In this subsection we prove Conjecture~\ref{main_conj} for rank 2 quivers. Again in this case, all positive real roots are Schur. The reflection corresponding to each positive real root is of the form $s_{i}s_{j}s_i\cdots s_is_{j}s_{i}$ with $\{i,j\}=\{1,2\}$. Cearly 
there exists a spiral $\eta$ (as an admissible curve) on the sphere $\Sigma_\sigma$ such that $\eta$ has no self-intersections and  $\upsilon(\eta)=iji\cdots iji$, where $\sigma$ is either $id$ or the transposition $(12)$.

\subsection{Rank 3 tame quiver}
Let $\mathcal{Q}$ be the rank 3 acyclic quiver of type $A_2^{(1)}$ as follows:
 
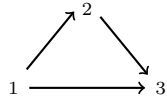
\begin{center}
 \begin{tikzpicture}[scale=0.7]
 \node at (0,0){\tiny{$1$}};
\node at (2.8,0){\tiny{$3$}};
\node at (1.4,1.5){\tiny{$2$}};
\draw [->, thick] (1.65,1.35)--(2.55,0.25);
\draw [->, thick] (0.25,0.35)--(1.15,1.45);
\draw [->, thick] (0.3,0)--(2.5,0);
\end{tikzpicture}
\end{center}
We have $P_{\mathcal{Q}}=\{id\}$. The positive real Schur roots are, for $n \ge 0$, 
\begin{align}
(n+1) \alpha_1 + n \alpha_2 + n \alpha_3 & = \begin{cases} (s_1s_2s_3s_2)^{\frac n 2} \alpha_1 & \text{ if $n$ is even}, \\  (s_1s_2s_3s_2)^{\frac {n-1} 2}s_1s_2 \alpha_3 & \text{ if $n$ is odd}, \end{cases}  \label{eqn-rs-1}\\
(n+1) \alpha_1 + (n+1) \alpha_2 + n \alpha_3 & = (s_1s_2s_3s_2)^n s_1 \alpha_2, \label{eqn-rs-2} \\ 
n \alpha_1 + (n+1) \alpha_2 + (n+1) \alpha_3 & = \begin{cases} (s_2s_3s_2s_1)^{\frac n 2} s_2\alpha_3 & \text{ if $n$ is even}, \\  (s_2s_3s_2s_1)^{\frac {n-1} 2} s_2 s_3 s_2\alpha_1  & \text{ if $n$ is odd}, \end{cases}, \label{eqn-rs-3} \\
n \alpha_1 + n \alpha_2 + (n+1) \alpha_3 & = (s_2s_3s_2s_1)^n s_2 s_3 \alpha_2, \label{eqn-rs-4} \\
\alpha_2 & \quad \text{ and } \quad \alpha_1+\alpha_3=s_1\alpha_3. \label{eqn-rs-5}
\end{align}
One can see that admissible curves corresponding to \eqref{eqn-rs-1}-\eqref{eqn-rs-4} are essentially determined by line segments with slopes $\frac n {n-1}, \frac {n+2}{n}, \frac {n+1}{n+2}, \frac {n+1} {n+3}$, respectively, on the universal cover of the torus with triangulation as in Example \ref{ex-main}. Curves corresponding to \eqref{eqn-rs-3} and \eqref{eqn-rs-4} are (isotopic to) line segments. When $n\ge 1$, curves corresponding to \eqref{eqn-rs-1} and \eqref{eqn-rs-2} revolve $180^\circ$ around a vertex  at the beginning, follow a line segment, and again revolve  $180^\circ$ around a vertex at the end. Clearly, such curves do not have self-intersections. The Schur roots $\alpha_2$ and $\alpha_1+\alpha_3$ trivially correspond to non-self-intersecting curves.

Conversely, an admissible curve with no self-intersections on a torus becomes a curve on the universal cover isotopic to a union of two spirals around vertices, which are symmetric to each other, and a line segment in the middle of the two spirals. It can be checked that each of such curves gives rise to one of the real Schur roots listed in  \eqref{eqn-rs-1}-\eqref{eqn-rs-5}.
Thus Conjecture \ref{main_conj} is verified in this case.

\section{Preliminaries} \label{Weyl}

\subsection{Cluster variables}
In this subsection, we review some notions from the theory of cluster algebras 
  introduced by Fomin and Zelevinsky in \cite{FZ}.
Our definition follows the exposition in \cite{FZ4}. For our purpose, it is enough to define the coefficient-free cluster algebras of rank 3.

We consider a field $\mathcal{F}$
isomorphic to the field of rational functions in $3$ independent
variables.

\begin{Def}
\label{def:seed}
A \emph{labeled seed} in~$\mathcal{F}$ is
a pair $(\mathbf{x}, B)$, where
\begin{itemize}
\item
$\mathbf{x} = (x_1, x_2, x_3)$ is a triple
from $\mathcal{F}$
forming a \emph{free generating set} over $\mathbb{Q}$,
 and
\item
$B = (b_{ij})$ is a $3\!\times\! 3$ integer \emph{skew-symmetric} matrix.
\end{itemize}
That is, $x_1, x_2, x_3$
are algebraically independent over~$\mathbb{Q}$, and
$\mathcal{F} =\mathbb{Q}(x_1, x_2, x_3)$.
We refer to~$\mathbf{x}$ as the (labeled)
\emph{cluster} of a labeled seed $(\mathbf{x}, B)$,
and to the
matrix~$B$ as the \emph{exchange matrix}.
\end{Def}

We  use the notation
$[x]_+ = \max(x,0)$
 and
\begin{align*}
\textup{sgn}(x) &=
\begin{cases}
x/|x| & \text{if $x\neq0$;}\\
0  & \text{if $x=0$.}
\end{cases}
\end{align*}

\begin{Def}
\label{def:seed-mutation}
Let $(\mathbf{x}, B)$ be a labeled seed in $\mathcal{F}$,
and let $k \in \{1,2,3\}$.
The \emph{seed mutation} $\mu_k$ in direction~$k$ transforms
$(\mathbf{x}, B)$ into the labeled seed
$\mu_k(\mathbf{x}, B)=(\mathbf{x}', B')$ defined as follows:
\begin{itemize}
\item
The entries of $B'=(b'_{ij})$ are given by
\begin{equation}
\label{eq:matrix-mutation}
b'_{ij} =
\begin{cases}
-b_{ij} & \text{if $i=k$ or $j=k$;} \\[.05in]
b_{ij} + \textup{sgn}(b_{ik}) \ [b_{ik}b_{kj}]_+
 & \text{otherwise.}
\end{cases}
\end{equation}
\item
The cluster $\mathbf{x}'=(x_1',x_2',x_3')$ is given by
$x_j'=x_j$ for $j\neq k$,
whereas $x'_k \in \mathcal{F}$ is determined
by the \emph{exchange relation}
\begin{equation}
\label{exchange relation}
x'_k = \frac
{\prod  x_i^{[b_{ik}]_+}
+ \ \prod x_i^{[-b_{ik}]_+}}{x_k} \, .
\end{equation}
\end{itemize}
\end{Def}


Let $B=(b_{ij})$ be a $3 \times 3$ skew-symmetric matrix, and $\mathcal Q$ be the rank $3$ acyclic quiver corresponding to $B$, with the set $I=\{1,2,3 \}$ of vertices, such that  the quiver $\mathcal{Q}$ has $b_{ij}$ arrows from $i$ to $j$ for $b_{ij}>0$. We will write $\xymatrix{i \ar[r]^{b_{ij}} &j}$ to represent the $b_{ij}$ arrows. Assume that $|b_{ij}| \ge 2$ for $i \neq j$.  Without loss of generality, we further assume that the vertex $1$ of $\mathcal Q$ is a source, the vertex $2$ a node, and the vertex $3$ a sink:
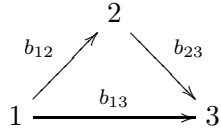
\[ \xymatrix{ & 2 \ar[rd]^{b_{23}}& \\ 1\ar[ru]^{b_{12}} \ar[rr]^{b_{13}} & & 3} \]

The matrix resulting from the mutation of  $B$ at the vertex $i \in I$, will be denoted by $B(i)$, and the matrix resulting from the mutation of $B(i)$ at the vertex $j \in I$ by $B(ij)$. Then the matrix $B(i_1i_2 \cdots i_k)$, $i_p \in I$ is the result of $k$ mutations. We will write $B(\mathfrak w)=(b_{ij}(\mathfrak w))$ for $\mathfrak w = i_1 \cdots i_k \in \mathfrak W$. When  $\mu_{\mathfrak w}$ is a sequence $(\mu_{i_1},\cdots, \mu_{i_k})$ of mutations performed from left to right, we write
$\Xi(\mathfrak w)= \mu_{\mathfrak w}(\Xi)= \mu_{i_k} \circ \cdots \circ \mu_{i_1} (\Xi)$  for a labeled seed $\Xi$.

The \emph{cluster variables} are the elements of clusters obtained  by sequences of seed mutations from the initial seed $((x_1,x_2,x_3),B)$.
 The remarkable {\it Laurent phenomenon} that  Fomin and Zelevinsky proved in \cite{FZ} implies the following:
\begin{Thm}
\label{Laurent}
Each cluster variable  is a Laurent polynomial over $\mathbb{Z}$ in the initial  cluster variables $x_1,x_2,x_3$.
\end{Thm}

Thanks to the Laurent Phenomenon, the denominator of every cluster variable is well defined when expressed in reduced form. 

\begin{Exa} \label{exa-mutation}
Let $\Xi = \left ((x_1,x_2,x_3), {\tiny \left(\begin{array}{ccc}0&2&2\\-2&0&2\\ -2&-2&0\end{array}  \right) }\right )$ be the initial seed. The mutation in direction $3$ yields 
\[ \Xi (3) =  {\tiny \left ( \left ( x_1,x_2, \tfrac {x_1^2 x_2^2+1}{x_3} \right ), \left(\begin{array}{ccc}0&2&-2\\-2&0&-2\\ 2&2&0\end{array}  \right) \right )}. \] Applying the mutation in direction $2$ to $\Xi(3)$, we obtain 
\[ \Xi (32) =  {\tiny \left ( \left ( x_1,\tfrac{x_1^2 (1+x_1^2 x_2^2)^2 + x_3^2}{x_2 x_3^2}, \tfrac {x_1^2 x_2^2+1}{x_3} \right ), \left(\begin{array}{ccc}0&-2&-2\\2&0&2\\ 2&-2&0\end{array}  \right) \right )}. \]
More mutations produce
$$
\Xi((321)^4)=\left(\left(\tfrac{P_1}{x_1^{4895}x_2^{12816}x_3^{33552}},\tfrac{P_2}{x_1^{1870}x_2^{4895}x_3^{12816}},\tfrac{P_3}{x_1^{714}x_2^{1870}x_3^{4895}}\right),\tiny{\left(\begin{array}{ccc}0&2&2\\-2&0&2\\ -2&-2&0\end{array}  \right) }\right),
$$
where the corresponding quiver is acyclic. We apply the mutation in direction $2$ to obtain
$$
\Xi((321)^42)=\left(\left(\tfrac{P_1}{x_1^{4895}x_2^{12816}x_3^{33552}},\tfrac{P_4}{x_1^{7920}x_2^{20737}x_3^{54288}},\tfrac{P_3}{x_1^{714}x_2^{1870}x_3^{4895}} \right),  \tiny{\left(\begin{array}{ccc}0&-2&6\\2&0&-2\\ -6&2&0\end{array}  \right) }\right),
$$
where the corresponding quiver is cyclic. We calculate three more mutations:
$$
\Xi((321)^423)=\left(\left(\tfrac{P_1}{x_1^{4895}x_2^{12816}x_3^{33552}},\tfrac{P_4}{x_1^{7920}x_2^{20737}x_3^{54288}},\tfrac{P_5}{x_1^{28656}x_2^{75026}x_3^{196417}}\right),\tiny{\left(\begin{array}{ccc}0&10&-6\\-10&0&2\\ 6&-2&0\end{array}  \right)}\right),
$$
$$
\Xi((321)^4231)=\left(\left(\tfrac{P_6}{x_1^{167041}x_2^{437340}x_3^{1144950}},\tfrac{P_4}{x_1^{7920}x_2^{20737}x_3^{54288}},\tfrac{P_5}{x_1^{28656}x_2^{75026}x_3^{196417}}\right),\tiny{\left(\begin{array}{ccc}0&-10&6\\10&0&-58\\ -6&58&0\end{array}  \right)}\right),
$$
$$
\Xi((321)^42312)=\left(\left(\tfrac{P_6}{x_1^{167041}x_2^{437340}x_3^{1144950}},\tfrac{P_7}{x_1^{1662490}x_2^{4352663}x_3^{11395212}},\tfrac{P_5}{x_1^{28656}x_2^{75026}x_3^{196417}}\right),\tiny{\left(\begin{array}{ccc}0&10&-574\\-10&0&58\\ 574&-58&0\end{array}  \right)}\right).
$$
Here all $P_i$ are polynomials in $x_1,x_2,x_3$ with no monomial factors. Compare \eqref{root-example} with the new cluster variable $P_7/x_1^{1662490}x_2^{4352663}x_3^{11395212}$ in $\Xi((321)^42312)$.
\end{Exa}

\subsection{Positive real roots}
Let $A=(a_{ij})$ be the corresponding symmetric Cartan matrix  to $B$ 
and $\mathfrak g$ be the associated Kac--Moody algebra. The simple roots of $\mathfrak g$  will be denoted by $\alpha_i$, $i \in I$.  Let $Q_+ = \oplus_{i \in I} \mathbb Z_{\ge 0} \alpha_i$ be the positive root lattice. We have the canonical bilinear form $(\cdot ,\cdot )$ on $Q_+$ defined by  $(\alpha_i, \alpha_j) = a_{ij}$ for $i,j \in I$.
  The simple reflections will be denoted by $s_i$, $i \in I$ and the Weyl group by $W$. As before, let $\mathfrak W$ be the set of words $\mathfrak w=i_1i_2 \cdots i_k$ in the alphabet $I$ such that no two consecutive letters $i_p$ and $i_{p+1}$ are the same. 
Since $W \cong (\mathbb Z/ 2 \mathbb Z) * (\mathbb Z/ 2 \mathbb Z) * (\mathbb Z/ 2 \mathbb Z)$,  we regard $\mathfrak W$  as the set of reduced expressions of the elements of $W$.  Assume that $\mathfrak W$ also has the empty word $\emptyset$.

For $\mathfrak w = i_1 \cdots i_k \in \mathfrak W$, define $$s_{\mathfrak w}  := s_{i_1} \cdots s_{i_k} \in W.$$ 
A root $\beta\in Q_+$ is called {\em real} if $\beta = w \alpha_i$ for some $w \in W$ and $\alpha_i$, $i \in I$. For a positive real root $\beta$, the corresponding reflection will be denoted by $r_\beta \in W$.

\section{Real Schur roots of rank 3 quivers}

In this section we prove Conjecture~\ref{main_conj} for rank 3 quivers with multiple arrows between every pair of vertices. We describe the set of real roots by the isotopy classes of certain curves on the universal cover of a triangulated torus and characterize the curves corresponding to real Schur roots.

\subsection{Curves representing real roots}
For easier visualization,    we    restate the set-up from Section~ \ref{conj_subsection} in terms of the universal cover. 
Consider the following set of lines on $\mathbb{R}^2$:
$$\mathcal{T}=\mathcal{T}_1\cup \mathcal{T}_2\cup \mathcal{T}_3,$$ 
where $\mathcal{T}_1=\{(x,y)\ : \ y\in\mathbb{Z}\}$,  $\mathcal{T}_2=\{(x,y)\ : \ x+y\in\mathbb{Z}\}$, and  $\mathcal{T}_3=\{(x,y)\ : \  x\in\mathbb{Z}\}$. 
Together with $\mathcal{T}$, the space $\mathbb{R}^2$ can be viewed as  the universal cover of a triangulated torus.

We also define \[ \mathcal L_1 =\{(x,y)\, : \, x-y\in\mathbb{Z}, \quad x -\lfloor x \rfloor < \tfrac 1 2 \} \quad \text{ and } \quad \mathcal L_2 =\{(x,y)\, : \, x-y\in\mathbb{Z}, \quad x -\lfloor x \rfloor > \tfrac 1 2 \} .\]

\begin{Def}
An \emph{admissible} curve is a continuous function $\eta:[0,1]\longrightarrow \mathbb{R}^2$ such that

1) $\eta(x)\in \mathbb Z^2$ if and only if  $x\in\{0,1\}$;

2) there exists $\epsilon>0$ such that $\eta([0,\epsilon])\subset \mathcal L_1$ and $\eta([1-\epsilon,1))\subset \mathcal L_2$;

3) if $\eta(x)\in \mathcal{T}\setminus \mathbb Z^2$ then $\eta([x-\epsilon,x+\epsilon])$ meets $\mathcal{T}$ transversally for sufficiently small $\epsilon>0$;

4) and   $\upsilon(\eta)\in R$, where $\upsilon(\eta):={i_1}\cdots {i_k}$ is given by 
$$\{x\in(0,1) \ : \ \eta(x)\in \mathcal{T}\}=\{x_1<\cdots<x_k\}\quad \text{ and }\quad \eta(x_\ell)\in \mathcal{T}_{i_\ell}\text{ for }\ell\in\{1,...,k\}.$$ 
\end{Def}

\subsection{Curves representing real Schur roots} \label{sec-rep}
Let
$$
\mathcal{Z}:=\{(a,b,c)\in\mathbb{Z}^3\ :  \ \text{gcd}(|b|,|c|)=1\}, 
$$
where $\text{gcd}(0,0)=\infty$ and $\text{gcd}(x,0)=x$ for nonzero $x$. Fix $z=(a,b,c)\in \mathcal{Z}$ and let $$\epsilon=\left\{\begin{array}{ll}1/2, & \text{ if }\max(|b|,|c|)=1;\\
1/2\sqrt{b^2+c^2}, & \text{ otherwise.} \end{array} \right.$$

Let $C_{z,1}\subset \mathbb{R}^2$ be the spiral that (i) crosses the  positive $x$-axis $|a|$ times; (ii) starts with the line segment from $(0,0)$ to $(\epsilon/2,\epsilon/2)$, goes around $(0,0)$, and ends at $(\epsilon b,\epsilon c)$; and (iii) revolves clockwise if $a>0$ (resp. counterclockwise if $a<0$). Let $C_{z,2}$ be the line segment from $(\epsilon b,\epsilon c)$ to $(b-\epsilon b,c-\epsilon c)$, and $C_{z,3}$ be the spiral obtained by rotating $C_{z,1}$ by $180^{\circ}$ around $(b/2,c/2)$.  Let $\eta_z$ be the union of $C_{z,1}$, $C_{z,2}$ and $C_{z,3}$.   We are ready to state our main theorem as follows.

Let $\Gamma$ be the set of (isotopy classes of) admissible curves $\eta$ such that $\eta$ has no     self-intersections on the torus, i.e., $\eta(x_1)=\eta(x_2)\ (\text{mod } \mathbb{Z}\times\mathbb{Z})$ implies $x_1=x_2$ or $\{x_1,x_2\}=\{0,1\}$. It is not hard to see that $\Gamma=\{\eta_z\ : \ z\in \mathcal{Z}\}$ by using Dehn twists. Recall that $\beta(\eta)$ is defined in Conjecture \ref{main_conj} for $\eta \in \Gamma$.

\begin{Thm} \label{thm-main}
The set $\{ \beta(\eta_z)\, : \,  z\in\mathcal{Z}\}$ is precisely the set of real Schur roots for $\mathcal{Q}$.
\end{Thm}

Clearly, the above theorem implies that Conjecture \ref{main_conj} holds for rank 3 quivers with multiple arrows between every pair of vertices. We will prove this theorem after we state Theorem \ref{thm-mm} below in Section \ref{subsec-clus}.

\subsection{Mutations of vectors and the definition of $\psi(\mathfrak w)$} \label{subsec-phi}
$\text{ }$\\
We define the triple $V(\mathfrak{w})$ of vectors on $\mathbb{R}^2$ for $\mathfrak w \in \mathfrak W\setminus\{\emptyset\}$ as follows.
First we define
\begin{align*}
V(1) &= (\langle -1,2 \rangle, \langle-1,1 \rangle, \langle0,1\rangle) ,\\
V(2) &= ( \langle 0,1 \rangle, \langle 1, 1 \rangle, \langle 1,0 \rangle ) ,\\
V(3) &= (\langle 1,0\rangle, \langle 1,-1 \rangle, \langle 2,-1 \rangle) .
\end{align*}
{Then} we inductively define $V(i_1\cdots i_q)$ for $q > 1$. Suppose that $V(i_1\cdots i_{q-1})=(\vec{v_1},\vec{v_2},\vec{v_3})$. Then $V(i_1\cdots i_q)$ is defined by
\begin{equation}\label{V_ind}
V(i_1\cdots i_q):=\left\{\begin{array}{ll}
(\vec{v_2}+\vec{v_3},\vec{v_2},\vec{v_3}), &\text{ if }i_q=1;\\  
(\vec{v_1},\vec{v_1}+\vec{v_3},\vec{v_3}), &\text{ if }i_q=2;\\  
(\vec{v_1},\vec{v_2},\vec{v_1}+\vec{v_2}), &\text{ if }i_q=3.\\  
 \end{array} \right.
\end{equation}

Write $V(\mathfrak{w})=(\vec{v_1}(\mathfrak w),\vec{v_2}(\mathfrak w),\vec{v_3}(\mathfrak w))$. Let $p\in\{1,2,3\}$ and suppose that $\vec{v_{p}}(\mathfrak w)=\langle b,c\rangle$. Let $\xi$ be (the isotopic class of) the line segment  from $(0,0)$ to $(b,c)$.  One can see that $\gcd(|b|,|c|)=1$ and so $\xi \in \Gamma$. We write 
\begin{equation} \label{eqn-ups} \beta(\vec{v_{p}}(\mathfrak w))= \beta(\xi) \ \text{ and } \ \upsilon(\vec{v_{p}}(\mathfrak w))= \upsilon(\xi)\in R.
\end{equation} 
For each $\mathfrak{w} \in\mathfrak W\setminus\{\emptyset\}$, we define a positive real root $\phi(\mathfrak w)$ by  
\[ \phi(\mathfrak w) := \beta(\vec{v_p}(\mathfrak w)),\]
where $p$ is the last letter of $\mathfrak{w} \in\mathfrak W\setminus\{\emptyset\}$.

\begin{Lem} For $\{i,j,k \} = \{1,2,3\}$, we have  
\begin{equation} \label{eqn-iji}  \phi ( (ij)^n) = (s_i s_j)^{n-1}s_i\alpha_j \quad \text{ and }  \quad \phi ( (ij)^ni) = (s_i s_j)^n \alpha_i ;\end{equation}
\begin{equation} \label{eqn-ijik}
\phi(i(ji)^n k)  = s_i (s_js_i)^{2n} \alpha_k \ \text{ and } \  \phi(i(ji)^n  j k)  = s_i (s_js_i)^{2n+1} \alpha_k ;
\end{equation}
\begin{align} \label{eqn-ijiki}
\phi(i(ji)^n ki) & = s_i (s_js_i)^{2n} s_k (s_is_j)^{n} \alpha_i, \\ \phi(i(ji)^n k j) & = s_i (s_js_i)^{2n} s_k (s_is_j)^{n} s_i \alpha_j,  \label{eqn-ijiki-1 } \\
\phi(i(ji)^n j k i) & = s_i (s_js_i)^{2n+1} s_k (s_is_j)^{n+1} \alpha_i , \label{eqn-ijiki-2} \\ \phi(i(ji)^n  j k j) &= s_i (s_js_i)^{2n+1} s_k (s_is_j)^{n} s_i \alpha_j. \label{eqn-ijiki-3}
\end{align}

\end{Lem}
\begin{proof}
It is straightforward to check these identities directly, so we omit the proof. Instead we illustrate the calculation for $\mathfrak w=121212$.

By \eqref{V_ind}, we get
$$
\begin{array}{ccc}
V(1)&=&(\langle-1,2\rangle, \langle-1,1 \rangle, \langle 0,1\rangle),\\
V(12)&=&(\langle-1,2\rangle, \langle-1,3 \rangle, \langle 0,1 \rangle),\\
\vdots & &\vdots\\
V(121212)&=&(\langle -1,6\rangle, \langle -1,7 \rangle, \langle 0,1 \rangle).
\end{array}
$$
Let $\xi$  be the line segment from $(0,0)$ to $(-1, 7)$.
Then $\upsilon(\xi)=(12)^51$, hence the corresponding real root is 
\[  \phi(121212)= \beta(\xi)= \beta(\langle-1,7 \rangle)= s_1s_2s_1s_2s_1\alpha_2. \]
\end{proof}

Let $\mathcal{C}_1=\{1,12,123,1231,12312,123123,...\}$ and $\mathcal{C}_3=\{3,32,321,3213,32132,321321,...\}\subset \mathfrak{W}$. Note that the quiver corresponding to $B(\mathfrak w)$ is acyclic if and only if $\mathfrak w\in\mathcal{C}_1\cup\mathcal{C}_3\cup\{\emptyset\}$. 
For $\mathfrak w=i_1\cdots i_k\in\mathfrak W$, let $\ell(\mathfrak w)=k$ and   $$\rho(\mathfrak w):=\left\{\begin{array}{ll}0, &\text{ if }\mathfrak w\text{ is the empty word }\emptyset,\text{ or }i_1=2;\\
\max\{p\, : \, i_1\cdots i_p\in \mathcal{C}_1\cup\mathcal{C}_3\}, &\text{ otherwise.}\end{array}\right.$$

The following definition is important for the rest of the paper.

\begin{Def}
Let $\hat{\mathfrak w} \in \mathfrak W \setminus \{ \emptyset \}$, and write 
$\hat{\mathfrak w}=\mathfrak w \mathfrak v \in \mathfrak W$ with the word $\mathfrak w$ being the longest word such that $B(\mathfrak w)$ is acyclic. Assume that $\mathfrak w = i_1 \dots i_k$. Then we have $k=\rho(\hat{\mathfrak w})$. Define a positive real root $\psi(\hat{\mathfrak w})$ by
\[ \psi(\hat{\mathfrak w}) := \begin{cases} s_{i_1} \cdots s_{i_{k-1}} \alpha_{i_k} & \text{if }\mathfrak v = \emptyset; \\ s_{\mathfrak w} \phi(\mathfrak v) & \text{otherwise}. \end{cases} \]
\end{Def}

\begin{Exa} \label{exa-sub}
Let $\mathcal{Q}$ be the following rank 3 acyclic quiver and $B$ be the corresponding skew-symmetric matrix:
 \begin{center}
 \raisebox{-0.5 cm}{\begin{tikzpicture}[scale=0.6]
 \node at (0,0){\tiny{$1$}};
\node at (2.8,0){\tiny{$3$}};
\node at (1.4,1.5){\tiny{$2$}};
\draw [->, thick] (1.6,1.3)--(2.5,0.2);
\draw [->, thick] (1.7,1.4)--(2.6,0.3);
\draw [->, thick] (0.2,0.3)--(1.1,1.4);
\draw [->, thick] (0.3,0.2)--(1.2,1.3);
\draw [->, thick] (0.3,0.05)--(2.5,0.05);
\draw [->, thick] (0.3,-0.1)--(2.5,-0.1);
\end{tikzpicture}}, \qquad $B={\scriptsize \left(\begin{array}{ccc}0&2&2\\-2&0&2\\ -2&-2&0\end{array}  \right)}$.
\end{center}

Consider $\hat{\mathfrak w} = (321)^42132 \in \mathfrak W$. Then $\mathfrak w=(321)^4$ and $\mathfrak v=2132$. One easily obtains \[ V(\mathfrak v)=V(2132) = (\langle 2,1 \rangle, \langle 5,3 \rangle, \langle 3,2 \rangle) .\] Thus $\vec{v_2}(\mathfrak v)= \langle 5,3 \rangle$. By recording the intersections of the line segment $\xi$ from $(0,0)$ to $(5,3)$ with $\mathcal T$, we have
$\upsilon(\xi)=2321232321232$ and $$\phi(\mathfrak v)=\beta(\xi) = s_2s_3s_2s_1s_2s_3\alpha_2.$$
Combining these, we obtain
\[ \psi(\hat{\mathfrak w})= s_{\mathfrak w} \phi(\mathfrak v) = (s_3s_2s_1)^4s_2s_3s_2s_1s_2s_3\alpha_2
= 1662490\alpha_1+4352663\alpha_2+11395212\alpha_3 .\]
This real root was considered in Example \ref{ex-main} (2). The word $\mathfrak w$ corresponds to the spirals and $\mathfrak v$ to the line segment  $\xi$.
\end{Exa}

\subsection{Denominators of cluster variables} \label{subsec-clus}

Consider the cluster variables associated to the initial seed $\Xi=((x_1, x_2, x_3), B)$. The denominator of a non-initial cluster variable will be identified with an element of the positive root lattice $Q_+$ through \begin{equation} \label{eqn-bij} x_1^{m_1}x_2^{m_2}x_3^{m_3} \longmapsto m_1 \alpha_1 + m_2 \alpha_2 +m_3 \alpha_3, \qquad m_i \in \mathbb Z_{\ge 0}, \quad i \in I . \end{equation} The denominators of the initial cluster variables $x_1, x_2, x_3$ correspond to $-\alpha_1, -\alpha_2, -\alpha_3$, respectively.

\begin{Thm}[\cite{CK}] \label{thm-CK}
The correspondence \eqref{eqn-bij} is a bijection between the set of denominators of cluster variables, other than $x_i$, $i \in I$, and the set of positive real Schur roots of $\mathcal Q$. 
\end{Thm}

For any $\mathfrak w \in \mathfrak W\setminus\{\emptyset\}$, let $\Xi(\mathfrak w)$ be the labeled seed obtained from the initial seed $\Xi$ by the sequence $\mu_{\mathfrak w}$ of mutations. We denote by $(\beta_1(\mathfrak w),\beta_2(\mathfrak w), \beta_3(\mathfrak w))$ the triple of real Schur roots (or negative simple roots) obtained from the denominators of the cluster variables in the cluster of $\Xi(\mathfrak w)$.

\begin{Exa}
In Example \ref{exa-mutation}, we obtain the triple of real Schur roots from $\Xi((321)^42312)$:
\begin{align*}
\beta_1((321)^42312)&=   167041 \, \alpha_1+{437340} \, \alpha_2 + {1144950}\,  \alpha_3, \\
\beta_2((321)^42312) &=1662490 \, \alpha_1 + {4352663}\,  \alpha_2 + {11395212} \, \alpha_3,  \quad \quad (\text{cf. Example}~\ref{exa-sub})\\
\beta_3((321)^42312) &={28656} \, \alpha_1+ {75026}\, \alpha_2 + {196417}\,  \alpha_3 .
\end{align*}
\end{Exa}

Now we state a description of the real Schur roots associated with the denominators of cluster variables, using sequences of simple reflections. 
\begin{Thm} \label{thm-mm}
Let $\hat{\mathfrak w} \in \mathfrak W\setminus\{\emptyset\}$.
If $p$ is the last letter of $\hat{\mathfrak w}$, then we have
\begin{equation} \label{eqn-mmm} \beta_p (\hat{\mathfrak w}) =  \psi(\hat{\mathfrak w}).
\end{equation}
\end{Thm}

This theorem will be proved in Section \ref{sec-proof}. Assuming this theorem, we now prove Theorem \ref{thm-main}.
\begin{proof}[Proof of Theorem \ref{thm-main}] 
By Theorems \ref{thm-CK} and \ref{thm-mm}, we have only to prove that there exists a one-to-one correspondence $\hat{\mathfrak w}=\mathfrak w \mathfrak v \in \mathfrak W \setminus \{ \emptyset \} \longmapsto z=(a,b,c) \in \mathcal Z$ such that $\psi(\hat{\mathfrak w}) = \beta (\eta_z)$, where the word $\mathfrak w$ is the longest word such that $B(\mathfrak w)$ is acyclic. By definition, we have $\mathfrak w \in \mathcal{C}_1\cup\mathcal{C}_3$, and it determines the spiral $C_{1}$ (and $C_{3}$) and the number $a$. Next consider the vector  $\vec{v_{p}}(\mathfrak v)=\langle b',c'\rangle$  and determine the sign for $\langle b,c \rangle =\pm \langle b', c' \rangle$ so that the line segment  $C_{2} \in \Gamma$ from $(\epsilon b,\epsilon c)$ to $(b-\epsilon b,c-\epsilon c)$  is connected to the spiral $C_{1}$ (and $C_{3}$) for sufficiently small $\epsilon>0$. Then we set $z=(a,b,c) \in \mathcal Z$ and define $\eta_z$ to be  the union of $C_{1}, C_{2}$ and $C_{3}$.

Conversely, given $z =(a,b,c) \in \mathcal Z$, we have the unique curve $\eta_z$ consisting of $C_{z,1}, C_{z,2}$ and $C_{z,3}$ by definition. The spiral $C_{z,1}$ determines $\mathfrak w \in \mathfrak W$ by simply recording the consecutive intersections of $C_{z,1}$ with $\mathcal T_p$, $p=1,2,3$. Since $\gcd (|b|,|c|)=1$, the line segment $C_{z,2}$ or the vector $\langle b, c \rangle$ determines a unique $\mathfrak v \in \mathfrak W$  such that $\vec{v_p}(\mathfrak v)= \pm \langle b, c \rangle$ where $p$ is the last letter of $\mathfrak v$.  Namely, one can associate  a Farey triple with $V(\mathfrak u)=(\vec{v_1}(\mathfrak u),\vec{v_2}(\mathfrak u),\vec{v_3}(\mathfrak u))$ by taking the ratio of two coordinates of each $\vec{v_i}(\mathfrak u)$, $i=1,2,3$, for each $\mathfrak u \in \mathfrak W$ and use the Farey tree (or the Stern--Brocot tree) to find $\mathfrak v$ (cf. \cite[pp. 52-53]{Ai}). Then we set $\hat{\mathfrak w} = \mathfrak w \mathfrak v$. 
This establishes the inverse of the map $\hat{\mathfrak w} \in \mathfrak W \setminus \{ \emptyset \} \longmapsto z \in \mathcal Z$.
\end{proof}

\section{Proof of Theorem \ref{thm-mm}} \label{sec-proof}

This section is devoted to a proof of Theorem \ref{thm-mm}. 
Recall that $\ell(\hat{\mathfrak w})=k$ and $\rho(\hat{\mathfrak w})=\max\{p\, : \, i_1\cdots i_p\in \mathcal{C}_1\cup\mathcal{C}_3\}$ for $\hat{\mathfrak w}=i_1\cdots i_k\in \mathfrak W$.  Note that  $\rho(\hat{\mathfrak w}) = \max \{ p \, : \, B(i_1 \dots i_p) \text{ is acyclic} \}$.
It is easy to check \eqref{eqn-mmm} if $\ell(\hat{\mathfrak w})=1$, so we assume that $\ell(\hat{\mathfrak w})\geq 2$.   Let  $$\delta(\hat{\mathfrak w}):=\max\{p\ : \  q+1\leq p\leq\ell(\hat{\mathfrak w})\text{ and }  i_{q}i_{q+1}\cdots i_p\text{ consists of two letters}\},$$where $q=\max(1,\rho(\hat{\mathfrak w})-1)$. We also let $\mathfrak w=i_1\cdots i_{\rho(\hat{\mathfrak w})}$.
 
We have $\ell(\hat{\mathfrak w})\ge\delta(\hat{\mathfrak w})\ge\rho(\hat{\mathfrak w})$ by definition. We plan to prove Theorem~\ref{thm-mm} by considering the following cases:
\begin{itemize}
\item[] Case 1: $\ell(\hat{\mathfrak w})=\delta(\hat{\mathfrak w})=\rho(\hat{\mathfrak w})$, \\
Case 2: $\ell(\hat{\mathfrak w})=\delta(\hat{\mathfrak w})=\rho(\hat{\mathfrak w})+1$, \\
Case 3: $\ell(\hat{\mathfrak w})=\delta(\hat{\mathfrak w})\geq\rho(\hat{\mathfrak w})+2$,\\
Case 4: $\ell(\hat{\mathfrak w})=\delta(\hat{\mathfrak w})+1$,\\
Case 5: $\ell(\hat{\mathfrak w})=\delta(\hat{\mathfrak w})+2$,\\
Case 6: $\ell(\hat{\mathfrak w})=\delta(\hat{\mathfrak w})+3$,\\
Case 7: $\ell(\hat{\mathfrak w})\geq\delta(\hat{\mathfrak w})+4$.\\
\end{itemize}

In what follows, we always set $\{i,j,k\}=\{1,2,3\}$.   Consider the natural partial order on $Q_{+}$, that is, $m_1\alpha_1+m_2\alpha_2+m_3\alpha_3   \ge m_1'\alpha_1+m_2'\alpha_2+m_3'\alpha_3$ if and only if $m_i\ge m_i'$ for all $i\in I$. We set $c_{ij} = |b_{ij}|$ for $i \neq j$ and $c_{ij}(\mathfrak v) = |b_{ij}(\mathfrak v)|$ for $i \neq j$ and $\mathfrak v \in \mathfrak W$.  

\subsection{Case 1: $\ell(\hat{\mathfrak w})=\delta(\hat{\mathfrak w})=\rho(\hat{\mathfrak w})$}
$\text{ }$\\
If $\ell(\hat{\mathfrak w})=\delta(\hat{\mathfrak w})=\rho(\hat{\mathfrak w})$ then $\hat{\mathfrak w}=\mathfrak w\in \mathcal{C}_1\cup\mathcal{C}_3$, 
equivalently $B(\hat{\mathfrak w})$ is acyclic. Write $\mathfrak w = \mathfrak u jk$ for $\mathfrak u \in \mathfrak W$ and $j, k \in I$. 
\begin{Lem} \label{lem-fff}
We have
  \begin{equation} \label{eqn-main-1}  \beta_k(\mathfrak u jk) = s_{\mathfrak u} s_j (\alpha_k)=\psi(\mathfrak u jk). \end{equation}
\end{Lem}

\begin{proof}
We have $\beta_u(u)=\alpha_u$ and $\beta_v(uv)= \alpha_v + c_{uv} \alpha_u = s_u(\alpha_v)$. Now, by induction,  we have 
\begin{align*}
\beta_j(\mathfrak u j) &= s_{\mathfrak u} (\alpha_j), \\ \beta_i(\mathfrak u j) & =\beta_i(\mathfrak u) = \beta_i (\mathfrak u' i) = s_{\mathfrak u'}(\alpha_i) = s_{\mathfrak u'} s_i s_i (\alpha_i) =s_{\mathfrak u} s_i (\alpha_i), \\
\beta_k(\mathfrak u j) & =\beta_k(\mathfrak u'' kij) = \beta_k (\mathfrak u'' k) = s_{\mathfrak u''}(\alpha_k) = s_{\mathfrak u''} s_k s_i s_i s_k (\alpha_k) =s_{\mathfrak u} s_i s_k (\alpha_k),
\end{align*}
where we write $\mathfrak u = \mathfrak u' i =\mathfrak u'' ki$.

Then  we have
\begin{align*}
\beta_k(\mathfrak u jk) &= - \beta_k(\mathfrak u j) + c_{ik}\beta_i(\mathfrak u j) + c_{jk}\beta_j(\mathfrak u j) \\ &= -s_{\mathfrak u} s_i s_k (\alpha_k) + c_{ik}s_{\mathfrak u} s_i (\alpha_i) + c_{jk}s_{\mathfrak u} (\alpha_j) \\ &= s_{\mathfrak u} [ s_i(\alpha_k) -c_{ik} \alpha_i + c_{jk} \alpha_j ] = s_{\mathfrak u} [ \alpha_k +c_{ik} \alpha_i -c_{ik} \alpha_i + c_{jk} \alpha_j ] \\ &= s_{\mathfrak u} [ \alpha_k + c_{jk} \alpha_j ] = s_{\mathfrak u} s_j(\alpha_k).
\end{align*}

Since $\psi(\mathfrak u jk)=s_{\mathfrak u} s_j (\alpha_k)$, we obtain
\begin{equation} \label{eqn-beta-0}
\beta_k(\mathfrak w) =\psi(\mathfrak w). 
\end{equation}
\end{proof}

\subsection{Case 2: $\ell(\hat{\mathfrak w})=\delta(\hat{\mathfrak w})=\rho(\hat{\mathfrak w})+1$}
Suppose that 
$\mathfrak w = \mathfrak u kij$. Then $\ell(\hat{\mathfrak w})=\delta(\hat{\mathfrak w})=\rho(\hat{\mathfrak w})+1$ implies that $\hat{\mathfrak w}=\mathfrak w i$.


\begin{Lem} \label{lem-ineq} Suppose that 
$\mathfrak w = \mathfrak u kij$ and $B(\mathfrak w)$ is acyclic. Then we have 
\[   c_{ij} \beta_j(\mathfrak w) \ge c_{ik} \beta_k(\mathfrak w) . \]
\end{Lem}

\begin{proof}
We use induction on the length of $\mathfrak w$. Base cases can be checked easily.  We have \[ \beta_i(\mathfrak u ki)=-\beta_i(\mathfrak u k)+ c_{ik}\beta_k(\mathfrak u k)+c_{ij} \beta_j(\mathfrak u k) \] and 
\begin{align*}
\beta_j(\mathfrak w) &= -\beta_j(\mathfrak u ki) + c_{jk} \beta_k(\mathfrak u ki) + c_{ij} \beta_i(\mathfrak u ki) \\ &=  -\beta_j(\mathfrak u ki) + c_{jk} \beta_k(\mathfrak u ki) + c_{ij} [-\beta_i(\mathfrak u k)+ c_{ik}\beta_k(\mathfrak u k)+c_{ij} \beta_j(\mathfrak u k)] \\ &=
(c_{jk}+c_{ij}c_{ik})\beta_k(\mathfrak u k) - c_{ij} \beta_i (\mathfrak u k) +(c_{ij}^2 -1) \beta_j(\mathfrak u k) .
\end{align*}

By induction, assume that $c_{jk}\beta_k(\mathfrak u k) \ge c_{ij}\beta_i(\mathfrak u k)$.  Then we have 
\begin{align*} c_{ij} \beta_j(\mathfrak w) &= c_{ij}(c_{jk}+c_{ij}c_{ik})\beta_k(\mathfrak u k) - b^2_{ij} \beta_i (\mathfrak u k) +c_{ij}(c_{ij}^2 -1) \beta_j(\mathfrak u k) \\ &\ge c_{ij}(c_{jk}+c_{ij}c_{ik})\beta_k(\mathfrak u k) - b^2_{ij} \beta_i (\mathfrak u k) \\ &\ge c_{ij}(c_{jk}+c_{ij}c_{ik})\beta_k(\mathfrak u k) - c_{ij}c_{jk} \beta_k (\mathfrak u k) =  b^2_{ij}c_{ik} \beta_k(\mathfrak u k) \\ & \ge c_{ik} \beta_k(\mathfrak u k) = c_{ik} \beta_k(\mathfrak w). \end{align*}   
\end{proof}

Suppose that $\mathfrak v \in \mathfrak W$ ends with $j$ and consider $\mathfrak v i$. If we have $c_{ij}(\mathfrak v) \beta_j(\mathfrak v) \ge c_{ik}(\mathfrak v) \beta_k(\mathfrak v)$, we record this situation using $[j]$  below the $i$-arrow in the following diagram:
\[ \xrightarrow[{\ \phantom{[i]} \ }]{j} \mathfrak v  \xrightarrow[{\ [j] \ }]{i} \mathfrak v i .\]
Similarly, if $c_{ij}(\mathfrak v) \beta_j(\mathfrak v) \le c_{ik}(\mathfrak v) \beta_k(\mathfrak v)$, we write
\[ \xrightarrow[{\ \phantom{[i]} \ }]{j} \mathfrak v  \xrightarrow[{\ [k] \ }]{i} \mathfrak v i .\]

Remember that $B(\mathfrak w)$ is acyclic and $B(\hat{\mathfrak w})=B(\mathfrak w i)$ is cyclic. By definition, we have $\psi(\mathfrak wi)=s_{\mathfrak w}\alpha_i$. If $\mathfrak w=\emptyset$, then $\beta_i(\mathfrak w i)= \beta_i(i)=\alpha_i$; if $\mathfrak w=j$, then $\beta_i(\mathfrak w i)=\beta_i(ji)=s_j(\alpha_i)$; if $\mathfrak w=ij$, then $\beta_i(\mathfrak w i)=\beta_i(iji)=s_is_j(\alpha_i)$. 
In all these cases, we have \eqref{eqn-mmm}.

Now suppose that $\mathfrak w = \mathfrak u kij$. 
By Lemma \ref{lem-ineq}, we have $c_{ij} \beta_j(\mathfrak w) \ge c_{ik} \beta_k(\mathfrak w)$. Thus we have
\[ \xrightarrow[{\ \phantom{[i]} \ }]{j} \mathfrak w  \xrightarrow[{\ [j] \ }]{i} \mathfrak w i .\]
By Case 1, we have 
\begin{align*} \beta_{i}(\mathfrak w i) &= -\beta_i(\mathfrak w)+c_{ij}\beta_j(\mathfrak w)= - s_{\mathfrak w} s_j s_i( \alpha _i) + c_{ij} s_{\mathfrak w} s_j (\alpha_j) \\
&= s_{\mathfrak w} [ s_j \alpha_i - c_{ij} \alpha_j] = s_{\mathfrak w } (\alpha_i).
 \end{align*} 
Thus we have
\begin{equation} \label{eqn-main-2}
\beta_{i}(\mathfrak w i) = s_{\mathfrak w } (\alpha_i)=\psi(\mathfrak wi). 
\end{equation}
This proves \eqref{eqn-mmm} in this case.

\subsection{Case 3: $\ell(\hat{\mathfrak w})=\delta(\hat{\mathfrak w})\geq\rho(\hat{\mathfrak w})+2$}
Assume that $B(\mathfrak w)$ is acyclic and $B(\mathfrak w i)$ is cyclic.

\begin{Lem} \label{lem-ini}
We have 
\[ \xrightarrow[{\ \phantom{[i]} \ }]{i} \mathfrak wi  \xrightarrow[{\ [i] \ }]{j} \mathfrak w ij  \xrightarrow[{\ [j] \ }]{i} \mathfrak w iji .\] That is, we have
\begin{align*} & c_{ij}(\mathfrak wi) \beta_i(\mathfrak w i) \ge c_{jk}(\mathfrak w i) \beta_k(\mathfrak wi), \quad \beta_j(\mathfrak w ij)=s_{\mathfrak w} s_i (\alpha_j), \\ \text{and \qquad} & c_{ij}(\mathfrak wij) \beta_j(\mathfrak w ij) \ge c_{ik}(\mathfrak w ij) \beta_k(\mathfrak wij),  \quad \beta_i(\mathfrak w iji)=s_{\mathfrak w} s_is_j (\alpha_i). \end{align*}
\end{Lem}

\begin{proof}
If the length of $\mathfrak w$ is less than $3$, it can be checked directly. Otherwise, write $\mathfrak w = \mathfrak u kij$. Using \eqref{eqn-main-2}, we have
\begin{align*} c_{ij}(\mathfrak wi) \beta_i(\mathfrak w i) &= c_{ij} s_{\mathfrak w} (\alpha_i) = c_{ij} s_{\mathfrak u} s_k s_i s_j (\alpha_i) \\
&= c_{ij} s_{\mathfrak u} [ (c_{ij}^2-1) \alpha_i + c_{ij} \alpha_j +(c_{ij}^2c_{ik} -c_{ik}+c_{ij}c_{jk}) \alpha_k ] \\ 
&=   c_{ij}(c_{ij}^2-1) s_{\mathfrak u}(\alpha_i) + c_{ij}^2 s_{\mathfrak u}(\alpha_j) +(c_{ij}^3c_{ik} -c_{ij}c_{ik}+c_{ij}^2c_{jk}) s_{\mathfrak u}(\alpha_k); \end{align*}
on the other hand, using \eqref{eqn-main-1}, we have
\begin{align*}
c_{jk}(\mathfrak w i) \beta_k(\mathfrak wi) &= (c_{jk}+c_{ij}c_{ik}) \beta_k (\mathfrak u k) = (c_{jk}+c_{ij}c_{ik}) s_{\mathfrak u}(\alpha_k).
\end{align*}
Since $\mathfrak u i$ and $\mathfrak u k$ are reduced expressions, we see that $s_{\mathfrak u}(\alpha_i)$ and $s_{\mathfrak u}(\alpha_k)$ are positive roots. We claim that $s_{\mathfrak u}(\alpha_i) \ge - s_{\mathfrak u}(\alpha_k)$. Indeed, writing $\mathfrak u = \mathfrak u' j$, we have
\begin{align}
s_{\mathfrak u}(\alpha_k) + s_{\mathfrak u}(\alpha_j) &= s_{\mathfrak u'}s_j(\alpha_k) +s_{\mathfrak u'}s_j(\alpha_j) = s_{\mathfrak u'}(\alpha_k +c_{jk} \alpha_j) -s_{\mathfrak u'}(\alpha_j) \nonumber \\ &= s_{\mathfrak u'}(\alpha_k)  +(c_{jk} -1) s_{\mathfrak u'}(\alpha_j)\ge 0, \label{eqn-uau}
\end{align}
since $s_{\mathfrak u'} (\alpha_k)$ is a positive root and $c_{jk} \ge 2$. 

Now we have only to show that
\[ - c_{ij}^2  +c_{ij}^3c_{ik} -c_{ij}c_{ik}+c_{ij}^2c_{jk} \ge c_{jk}+c_{ij}c_{ik} ,\]
which is equivalent to
\[ c_{ij}^2c_{jk}-c_{ij}^2 -c_{jk} + c_{ij}^3c_{ik} - 2 c_{ij}c_{ik} \ge 0 .\]
We write the left-hand side of the inequality as
\[ (c_{ij}^2 -1)(c_{jk}-1) -1+ (c_{ij}^2-2)  c_{ij}c_{ik}, \]
and we are done since $c_{ij}, c_{jk} \ge 2$. 

Note that $c_{ij}(\mathfrak w i)=(\beta_i(\mathfrak w i), \beta_j(\mathfrak w i))$. Indeed, since $\beta_i({\mathfrak w} i)=s_{{\mathfrak w}} (\alpha_i)$ and $\beta_j(\mathfrak w i) = s_{\mathfrak w} s_j (\alpha_j)$, we have
\[  (\beta_i(\mathfrak w i), \beta_j(\mathfrak w i)) = (s_{\mathfrak w}(\alpha_i), s_{\mathfrak w} s_j (\alpha_j) )= - (\alpha_i, \alpha_j) =c_{ij}= c_{ij}(\mathfrak w i) .\]
Then, 
from $r_{\beta_i({\mathfrak w} i)} = s_{{\mathfrak w}} s_i s_{{\mathfrak w}}^{-1}$, we obtain
\begin{align} \label{eqn-main-3} \beta_{j}({\mathfrak w} ij) &= -\beta_j(\mathfrak w i) + c_{ij}(\mathfrak w i) \beta_i(\mathfrak w i) =  -\beta_j(\mathfrak w i) + (\beta_i(\mathfrak w i), \beta_j(\mathfrak w i)) \beta_i(\mathfrak w i) \\ &=
-r_{\beta_i({\mathfrak w} i)} (\beta_j({\mathfrak w} i) ) = -s_{{\mathfrak w}} s_i s_{{\mathfrak w}}^{-1} s_{{\mathfrak w}} s_j  (\alpha_j) = s_{{\mathfrak w}} s_i(\alpha_j). \nonumber
\end{align}

A similar argument establishes $c_{ij}(\mathfrak wij) \beta_j(\mathfrak w ij) \ge c_{ik}(\mathfrak w ij) \beta_k(\mathfrak wij)$ and $\beta_i(\mathfrak w iji)=s_{\mathfrak w} s_is_j (\alpha_i)$. 
\end{proof}

\begin{Lem} \label{ppa}
Let $\tilde{\mathfrak w}= \mathfrak w i(ji)^n$ for $n \in \mathbb Z_{\ge 0}$. Then we have
\begin{align} c_{ij}(\tilde{\mathfrak w}) \beta_i(\tilde{\mathfrak w}) & > c_{jk}(\tilde{\mathfrak w}) \beta_k(\tilde{\mathfrak w}), \quad \beta_j(\tilde{\mathfrak w} j)=s_{\tilde{\mathfrak w}}  (\alpha_j), \label{eqn-fi} \\ c_{ij}(\tilde{\mathfrak w}j) \beta_j(\tilde{\mathfrak w}j) & > c_{ik}(\tilde{\mathfrak w}j) \beta_k(\tilde{\mathfrak w}j), \quad \beta_i(\tilde{\mathfrak w}ji)=s_{\tilde{\mathfrak w}} s_j (\alpha_i). \label{eqn-se} \end{align} 
\end{Lem}
This means that 
 we have \[ \mathfrak w i(ji)^n \xrightarrow[{\ [i] \ }]{j} \mathfrak w i(ji)^nj \xrightarrow[{\ [j] \ }]{i} \mathfrak w i(ji)^{n+1} \]
for each $n \in \mathbb Z_{\ge 0}$.

\begin{proof}
We have $\beta_k(\mathfrak w )=\beta_k(\mathfrak w i(ji)^n)= \beta_k(\mathfrak w i(ji)^nj)$. Similarly, $c_{ij}=c_{ij}(\mathfrak w i)= c_{ij}(\mathfrak w i(ji)^n)= c_{ij}(\mathfrak w i(ji)^nj)$. We write $\gamma = c_{ij}$ for simplicity. We use induction on $n$. The case $n=0$ is proven in Lemma \ref{lem-ini}. Thus we assume $n>0$. If we consider the vector $( \beta_i(\mathfrak w'), \beta_j(\mathfrak w')$ for $\mathfrak w' = \mathfrak w i(ji)^m$ with $m<n$,  the vector $( \beta_i(\mathfrak w'j), \beta_j(\mathfrak w'j)$ after mutation $j$ is given by the matrix $J_j:=\begin{pmatrix} 1 & \gamma \\ 0 & -1 \end{pmatrix}$ through right multiplication, and the vector after mutation $i$ for $\mathfrak w' = \mathfrak w i(ji)^mj$ is given by $J_i := \begin{pmatrix} -1 & 0 \\ \gamma & 1 \end{pmatrix}$.
 Similarly, if we consider the vector $(c_{jk}(\mathfrak w') , c_{ik}(\mathfrak w') )$, the matrices for mutation $j$ and $i$ are respectively given by the same matrices $J_j$ and $J_i$.

Let $J=J_jJ_i= \begin{pmatrix} \gamma^2-1 & \gamma \\ -\gamma & -1 \end{pmatrix}$.
First, assume that $\gamma>2$.
We 
denote two eigenvalues of $J$ by $\lambda_1$ and $\lambda_2$ with $\lambda_1 > \lambda_2$. 
Then we have $\lambda_1 > 1 > \lambda_2>0$. 
A diagonalization  $J=PDP^{-1}$ of $J$ is given by $D=\begin{pmatrix} \lambda_1 & 0 \\ 0 & \lambda_2 \end{pmatrix}$ and $P=\begin{pmatrix} 1+\lambda_1 & 1+ \lambda_2 \\ - \gamma & - \gamma \end{pmatrix}$. With $\lambda_1+\lambda_2=\gamma^2-2$ and $\lambda_1 \lambda_2=1$, we compute to obtain
\[ J^n= \tfrac 1 {\lambda_1 - \lambda_2}  \begin{pmatrix} \lambda_1^n(1+\lambda_1) -  \lambda_2^n(1+\lambda_2) & \gamma ( \lambda_1^n - \lambda_2^n )  \\ -\gamma ( \lambda_1^n - \lambda_2^n) & - \lambda_1^{n-1}(1+\lambda_1) +  \lambda_2^{n-1}(1+\lambda_2) \end{pmatrix} .\] 

We let 
\begin{align*}
y_n & :=  \tfrac 1 {\lambda_1 - \lambda_2} ( \lambda_1^n - \lambda_2^n ) =  \lambda_1^{n-1} + \lambda_1^{n-3} + \cdots + \lambda_2^{n-3} +\lambda_2^{n-1} , \\
x_n & :=  \tfrac 1 {\lambda_1 - \lambda_2} (\lambda_1^n(1+\lambda_1) -  \lambda_2^n(1+\lambda_2)) = y_{n+1}+y_n.
\end{align*}
Then we have
$J^n= \begin{pmatrix} x_n & \gamma y_n \\ -\gamma y_n & -x_{n-1} \end{pmatrix}$ for $n>0$.

Next, assume that $\gamma =2$, and let $y_n= n$ and $x_n=2n+1$. Then, from direct computation, we have
\[ J^n = \begin{pmatrix} 2n+1 & 2n \\ - 2n & -2n+1 \end{pmatrix} = \begin{pmatrix} x_n & \gamma y_n \\ -\gamma y_n & -x_{n-1} \end{pmatrix}  \quad \text{ for } n>0, \]
and the same formula for $J^n$ holds in this case as well.

We want to prove \eqref{eqn-fi}, which can be written as
\begin{align*} \gamma \beta_i(\mathfrak w i(ji)^n) & > c_{jk}(\mathfrak w i(ji)^n) \beta_k(\mathfrak w ).  \end{align*} 
If the length of $\mathfrak w$ is less than $3$, i.e. $\mathfrak w = \emptyset, j$ or $ij$, then $\beta_k(\mathfrak w) =0$ and there is nothing to prove. Thus we assume $\mathfrak w = \mathfrak u kij$.

Using the matrices $J^n$, we rewrite the inequality as
\begin{align} \label{eqn-fde} \gamma (x_n \beta_i(\mathfrak w i) - \gamma y_n \beta_j(\mathfrak w i)) & > (x_n c_{jk}(\mathfrak w i) -\gamma y_n c_{ik}(\mathfrak w i)) \beta_k(\mathfrak w ),  \end{align} 
which becomes 
\begin{equation} \label{eqn-su} \gamma ( x_n s_{\mathfrak u} s_ks_is_j(\alpha_i) -\gamma y_n s_{\mathfrak u}s_ks_i(\alpha_j)) > (x_n(c_{jk}+\gamma c_{ik}) -\gamma y_n c_{ik}) s_{\mathfrak u}(\alpha_k) .\end{equation}

We expand each side of \eqref{eqn-su}.
\begin{align*}
\text{LHS} &= \gamma x s_{\mathfrak u} [ (\gamma^2 -1) \alpha_i + \gamma \alpha_j + (\gamma^2 c_{ik} -c_{ik}+ \gamma c_{jk}) \alpha_k) ] \\ & \hskip 1 cm - \gamma^2 y_n s_{\mathfrak u} [ \gamma \alpha_i + \alpha_j + (c_{jk}+ \gamma c_{ik}) \alpha_k ] \\
&= [\gamma x_n (\gamma^2 -1) - \gamma^3 y_n] s_{\mathfrak u}(\alpha_i) + [\gamma^2 x_n - \gamma^2 y_n] s_{\mathfrak u}(\alpha_j) \\
& \hskip 1 cm + [\gamma x_n(\gamma^2 c_{ik} - c_{ik} + \gamma c_{jk}) - \gamma^2 y_n(c_{jk}+ \gamma c_{ik})] s_{\mathfrak u}(\alpha_k),\\
\text{RHS} &= [x_n(c_{jk}+ \gamma c_{ik}) - \gamma y_n c_{ik}] s_{\mathfrak u}(\alpha_k).
\end{align*}
We consider the coefficient of $s_{\mathfrak u}(\alpha_i)$ in LHS and find
\[ x_n(\gamma^2 -1) -\gamma^2 y_n =(y_{n+1}+y_n)(\gamma^2 -1) -\gamma^2 y_n = (\gamma^2 -1) y_{n+1} -y_n \ge 0 \]
since $\gamma \ge 2$ and $y_{n+1} > y_n$.

Recall that we showed $s_{\mathfrak u}(\alpha_j) \ge - s_{\mathfrak u}(\alpha_k)$ in \eqref{eqn-uau}. We combine the coefficients of $s_{\mathfrak u}(\alpha_j)$ and $s_{\mathfrak u}(\alpha_k)$ in LHS and need to prove the following inequality.
\begin{align*}
-\gamma^2 x_n + \gamma^2 y_n + \gamma^3 x_n c_{ik}  - \gamma x_n c_{ik} + \gamma^2 x_n c_{jk} -\gamma^2 y_n c_{jk} -\gamma^3 y_n c_{ik}  \ge  x_n c_{jk} + \gamma x_n c_{ik} - \gamma y_n c_{ik}.
\end{align*}
With $x_n=y_{n+1}+y_n$ substituted, the inequality is equivalent to
\begin{equation} \label{eqn-uu} -\gamma^2 y_{n+1} + (\gamma^3 -2 \gamma) y_{n+1}c_{ik} -\gamma y_{n} c_{ik} + (\gamma^2-1) y_{n+1} c_{jk} - y_{n} c_{jk} \ge 0 . \end{equation}
Using $\gamma, c_{ik}, c_{jk} \ge 2$ and $y_{n+1} > y_n$,
we see that
\begin{align*}
& -\gamma^2 y_{n+1} + (\gamma^3 -2 \gamma) y_{n+1}c_{ik} -\gamma y_{n} c_{ik} + (\gamma^2-1) y_{n+1} c_{jk} - y_{n} c_{jk} \\ \ge & -\gamma^2 y_{n+1} + (\gamma^3 -3 \gamma) y_{n+1}c_{ik}  + (\gamma^2-2) y_{n+1} c_{jk} \\ \ge & -\gamma^2 y_{n+1} + 2(\gamma^3 -3 \gamma) y_{n+1}  + 2(\gamma^2-2) y_{n+1}  =\gamma (2\gamma^2+\gamma -10) y_{n+1} \ge 0.
\end{align*}
Thus the inequality \eqref{eqn-uu}  is proven, so is the inequality \eqref{eqn-fi}.

One can see that $c_{ij}(\tilde{\mathfrak w})=c_{ij} = (\beta_i(\tilde{\mathfrak w}), \beta_j(\tilde{\mathfrak w}))$ by induction. Then a similar computation to \eqref{eqn-main-3} gives us $\beta_j(\tilde{\mathfrak w} j)=s_{\tilde{\mathfrak w}}  (\alpha_j)$.

Now we want to prove 
\begin{align*} \gamma \beta_j(\mathfrak w i(ji)^{n}j) & > c_{ik}(\mathfrak w i(ji)^{n}j) \beta_k(\mathfrak w ),  \end{align*} which is the same as \eqref{eqn-se}.
Since 
\[ J^nJ_j = \begin{pmatrix} x_n & \gamma y_n \\ -\gamma y_n & -x_{n-1} \end{pmatrix} \begin{pmatrix} 1 & \gamma \\ 0 & -1 \end{pmatrix} = \begin{pmatrix} x_n & \gamma y_{n+1} \\ -\gamma y_n & - x_n \end{pmatrix} ,\] 
the inequality can be written as
\begin{align*}  \gamma (\gamma y_{n+1}  \beta_i(\mathfrak w i) - x_n \beta_j(\mathfrak w i)) &> (\gamma y_{n+1}  c_{jk}(\mathfrak w i) -x_n c_{ik}(\mathfrak w i)) \beta_k(\mathfrak w ),  \end{align*} 
and  becomes
\begin{align*}  \gamma (\gamma y_{n+1}  s_{\mathfrak u} s_ks_is_j(\alpha_i) - x_n s_{\mathfrak u}s_ks_i(\alpha_j)) &> (\gamma y_{n+1}  (c_{jk}+\gamma c_{ik}) -x_n c_{ik}) s_{\mathfrak u}(\alpha_k).  \end{align*} 
This can be proven in the same way as we did for \eqref{eqn-su}. Similarly, we obtain $\beta_i(\tilde{\mathfrak w}ji)=s_{\tilde{\mathfrak w}} s_j (\alpha_i)$.

\end{proof}

Let $\tilde{\mathfrak w}= \mathfrak w i(ji)^n$ for $n \in \mathbb Z_{\ge 0}$ and $\mathfrak v = i(ji)^n$. By \eqref{eqn-iji}, we have \[ \phi ( \mathfrak v j) = s_i (s_js_i)^n \alpha_j  \ \text{ and } \ \phi ( \mathfrak v ji) = s_i (s_js_i)^n s_j \alpha_i,\] 
and obtain
\[ \beta_j (\tilde{\mathfrak w}j )  = s_{\tilde{\mathfrak w}} (\alpha_j) = s_{\mathfrak w}\phi ({\mathfrak v} j )= \psi (\mathfrak w{\mathfrak v} j )=\psi(\tilde{\mathfrak w}j)  \ \text{ and } \ \beta_i (\tilde{\mathfrak w}ji )  = s_{\tilde{\mathfrak w}} s_j (\alpha_i) =\psi(\tilde{\mathfrak w}ji) . \]
Thus we have proven \eqref{eqn-mmm} in this case.

\subsection{Case 4: $\ell(\hat{\mathfrak w})=\delta(\hat{\mathfrak w})+1$}

\begin{Lem} \label{pqa}
Let $\tilde{\mathfrak w}= \mathfrak w i(ji)^n$ for $n \in \mathbb Z_{\ge 0}$. Then we have
\begin{align} c_{jk}(\tilde{\mathfrak w}) \beta_j(\tilde{\mathfrak w}) & > c_{ik}(\tilde{\mathfrak w}) \beta_i(\tilde{\mathfrak w}), \quad \beta_k(\tilde{\mathfrak w} k)=s_{\tilde{\mathfrak w}} (s_js_i)^n (\alpha_k),  \\ c_{ik}(\tilde{\mathfrak w}j) \beta_i(\tilde{\mathfrak w}j) & > c_{jk}(\tilde{\mathfrak w}j) \beta_j(\tilde{\mathfrak w}j), \quad \beta_k(\tilde{\mathfrak w}jk)=s_{\tilde{\mathfrak w}} (s_js_i)^{n+1} (\alpha_k).  \end{align} 
\end{Lem}
This means that 
 we have \[  \xrightarrow[{\ [j] \ }]{i} \tilde{\mathfrak w} \xrightarrow[{\ [j] \ }]{k} \tilde{\mathfrak w}k \quad \text{ and } \quad  \xrightarrow[{\ [i] \ }]{j} \tilde{\mathfrak w}j \xrightarrow[{\ [i] \ }]{k} \tilde{\mathfrak w}jk .\]

\begin{proof} 

First, we consider the case $n=0$ and see
\begin{align*}
c_{jk}(\mathfrak w i) \beta_j(\mathfrak w i) &= (c_{jk}+c_{ij}c_{ik}) \beta_j(\mathfrak w) = c_{jk} \beta_j(\mathfrak w) + c_{ij}c_{ik}\beta_j(\mathfrak w), \\
c_{ik}(\mathfrak w i) \beta_i(\mathfrak w i) &= c_{ik}( -\beta_i(\mathfrak w) + c_{ij}\beta_j(\mathfrak w) ) = - c_{ik} \beta_i(\mathfrak w) + c_{ik}c_{ij}\beta_j(\mathfrak w).
\end{align*}
Thus we have 
\begin{equation} \label{eqn-true}
c_{jk}(\mathfrak w i)\beta_j(\mathfrak w i) > c_{ik}(\mathfrak w i) \beta_i(\mathfrak w i).
\end{equation} 
We claim that $c_{jk}(\mathfrak w i) = (\beta_k(\mathfrak w i), \beta_j(\mathfrak w i))$. Indeed, if the length of $\mathfrak w $ is greater than $3$, we have
\begin{align*} (\beta_k(\mathfrak w i), \beta_j(\mathfrak w i)) &= (s_{\mathfrak w} s_j s_i s_k (\alpha_k), s_{\mathfrak w} s_j (\alpha_j))= (s_is_k(\alpha_k), \alpha_j) \\ &=-(\alpha_k + c_{ik}\alpha_i, \alpha_j) = c_{jk} + c_{ij}c_{ik} = c_{jk}(\mathfrak w i). \end{align*} Otherwise, it can be checked easily.
Then we obtain
\begin{align*} 
\beta_k(\mathfrak w ik)&=-\beta_k(\mathfrak w i) +c_{jk}(\mathfrak w i) \beta_j(\mathfrak w i)
= -\beta_k(\mathfrak w i) +(\beta_k(\mathfrak w i), \beta_j(\mathfrak w i)) \beta_j(\mathfrak w i) \\ & = -r_{\beta_j(\mathfrak w i)}(\beta_k(\mathfrak w i)) = -s_{\mathfrak w} s_j s_{\mathfrak w}^{-1} \, s_{\mathfrak w} s_j s_i s_k (\alpha_k) = s_{\mathfrak w} s_i (\alpha_k).
\end{align*}

Now assume $n>0$. Using the matrix $J^n= \begin{pmatrix} x_n & \gamma y_n \\ -\gamma y_n & -x_{n-1} \end{pmatrix}$ defined in the proof of Lemma \ref{ppa}, the inequality  $c_{jk}(\tilde{\mathfrak w}) \beta_j(\tilde{\mathfrak w}) > c_{ik}(\tilde{\mathfrak w}) \beta_i(\tilde{\mathfrak w})$ can be written as
\begin{align*}  & (x_n c_{jk}(\mathfrak w i) - \gamma y_n c_{ik}(\mathfrak w i)) ( \gamma y_n \beta_i(\mathfrak w i) - x_{n-1} \beta_j(\mathfrak w i))\\ & \hskip 3 cm > (\gamma y_n c_{jk}(\mathfrak w i) - x_{n-1} c_{ik}(\mathfrak w i)) ( x_n \beta_i(\mathfrak w i) -\gamma y_n \beta_j(\mathfrak w i)), \end{align*}
which is equivalent to 
\[ (\gamma^2 y_n^2 -x_n x_{n-1}) c_{jk}(\mathfrak w i) \beta_j(\mathfrak w i) > (\gamma^2 y_n^2 -x_n x_{n-1}) c_{ik}(\mathfrak w i) \beta_i(\mathfrak w i) .\]
Since $\gamma^2 y_n^2 -x_n x_{n-1}=\det J^n=1$, this inequality is the same as \eqref{eqn-true} and we are done. 

We claim that $c_{jk}(\tilde{\mathfrak w}) = (\beta_j(\tilde{\mathfrak w}), \beta_k(\tilde{\mathfrak w}))$. Indeed, we have
\begin{align*} c_{jk}(\tilde{\mathfrak w}) & = x_n c_{jk}(\mathfrak w i) - \gamma y_n c_{ik}(\mathfrak w i) \ \text{ and } \\ (\beta_j(\tilde{\mathfrak w}), \beta_k(\tilde{\mathfrak w})) & = (\gamma y_n \beta_i (\mathfrak w i) - x_{n-1} \beta_j(\mathfrak w i), \beta_k(\mathfrak w i))= \gamma y_n (\beta_i(\mathfrak w i), \beta_k (\mathfrak w i))-x_{n-1} c_{jk}(\mathfrak w i). \end{align*}
Since $x_n+x_{n-1}=\gamma^2 y_n$, $c_{jk}(\mathfrak w i) = c_{jk}+\gamma c_{ik}$ and $(\beta_i(\mathfrak w i), \beta_k (\mathfrak w i))= -c_{ik}+ \gamma c_{jk} + \gamma^2 c_{ik}$, one sees that the claim holds.
Then  we obtain
\begin{align*}
\beta_k(\tilde{\mathfrak w} k) &= -\beta_k(\tilde{\mathfrak w}) + c_{jk}(\tilde{\mathfrak w})\beta_j(\tilde{\mathfrak w}) = -\beta_k(\tilde{\mathfrak w}) + (\beta_j(\tilde{\mathfrak w}), \beta_k(\tilde{\mathfrak w})) \beta_j(\tilde{\mathfrak w}) = -r_{\beta_j(\tilde{\mathfrak w})}(\beta_k(\tilde{\mathfrak w})) \\ & = -s_{\mathfrak w} s_i (s_js_i)^{n-1} s_j (s_is_j)^{n-1}s_i s_{\mathfrak w}^{-1} s_{\mathfrak u}(\alpha_k) = -s_{\mathfrak w} s_i (s_js_i)^{n-1} s_j (s_is_j)^{n-1}s_i s_js_is_ks_{\mathfrak u}^{-1}  s_{\mathfrak u}(\alpha_k) \\ &= s_{\mathfrak w} s_i (s_js_i)^{n-1} s_j (s_is_j)^{n-1}s_i s_js_i (\alpha_k) = s_{\tilde{\mathfrak w}} (s_js_i)^n (\alpha_k),
\end{align*}
where we write $\mathfrak w = \mathfrak u kij$ as before. 

Similarly, the inequality $c_{ik}(\tilde{\mathfrak w}j) \beta_i(\tilde{\mathfrak w}j) > c_{jk}(\tilde{\mathfrak w}j) \beta_j(\tilde{\mathfrak w}j) $
can be proven in the same way, using the matrix $J^nJ_j = \begin{pmatrix} x_n & \gamma y_{n+1} \\ -\gamma y_n & - x_n \end{pmatrix}$ and $\det(J^n J_j) =-1$. Furthermore, we see that 
\[ c_{ik}(\tilde{\mathfrak w} j) = ( \beta_i(\tilde{\mathfrak w}j), \beta_k(\tilde{\mathfrak w}j)) \] and obtain
\begin{align*}
\beta_k(\tilde{\mathfrak w}j k) &= -r_{\beta_i(\tilde{\mathfrak w}j)}(\beta_k(\tilde{\mathfrak w})j) = -s_{\mathfrak w} s_i (s_js_i)^{n-1} s_j s_i s_j (s_is_j)^{n-1}s_i s_{\mathfrak w}^{-1} s_{\mathfrak u}(\alpha_k) \\
&= s_{\mathfrak w} s_i (s_js_i)^{n-1} s_j s_i s_j (s_is_j)^{n-1}s_i s_js_i (\alpha_k) = s_{\tilde{\mathfrak w}} (s_js_i)^{n+1} (\alpha_k).
\end{align*} 

\end{proof}

Let $\tilde{\mathfrak w}= \mathfrak w i(ji)^n$ for $n \in \mathbb Z_{\ge 0}$ and $\mathfrak v = i(ji)^n$. By \eqref{eqn-ijik}, we have \[ \phi ( \mathfrak v k) = s_i (s_js_i)^{2n} \alpha_k  \ \text{ and } \ \phi ( \mathfrak v jk) = s_i (s_js_i)^{2n+1} \alpha_k,\] 
and obtain
\[ \beta_k (\tilde{\mathfrak w}k )  = s_{\tilde{\mathfrak w}} (s_js_i)^n (\alpha_k) = s_{\mathfrak w} \phi ({\mathfrak v} k )=\psi(\tilde{\mathfrak w}k)  \ \text{ and } \ \beta_k (\tilde{\mathfrak w}jk ) = s_{\tilde{\mathfrak w}} (s_js_i)^{n+1} (\alpha_k) = \psi(\tilde{\mathfrak w}jk). \]
Thus we have proven \eqref{eqn-mmm} in this case.

Before we go to the next case, we list the values of the bilinear form for various roots. Some of them have already been  proved in the proof of Lemma \ref{pqa}. As the others can be easily checked, we omit the details.
\begin{Cor} \label{cor-var}
We have
\begin{align*}
(\beta_i(\tilde{\mathfrak w}), \beta_j(\tilde{\mathfrak w}))&=c_{ij}(\tilde{\mathfrak w}),
& (\beta_i(\tilde{\mathfrak w}j), \beta_j(\tilde{\mathfrak w}j))&=c_{ij}(\tilde{\mathfrak w}j),\\
(\beta_i(\tilde{\mathfrak w}), \beta_k(\tilde{\mathfrak w}))&=-c_{ik}(\tilde{\mathfrak w})+c_{ij}(\tilde{\mathfrak w})c_{jk}(\tilde{\mathfrak w}),
& (\beta_i(\tilde{\mathfrak w}j), \beta_k(\tilde{\mathfrak w}j))&=c_{ik}(\tilde{\mathfrak w}j),\\
(\beta_j(\tilde{\mathfrak w}), \beta_k(\tilde{\mathfrak w}))&=c_{jk}(\tilde{\mathfrak w}),
& (\beta_j(\tilde{\mathfrak w}j), \beta_k(\tilde{\mathfrak w}j))&=-c_{jk}(\tilde{\mathfrak w}j)+c_{ij}(\tilde{\mathfrak w}j)c_{ik}(\tilde{\mathfrak w}j).
\end{align*}
\end{Cor}

\subsection{Case 5: $\ell(\hat{\mathfrak w})=\delta(\hat{\mathfrak w})+2$}

\begin{Lem} \label{lem-wwaa}
We have
\[ \xymatrix{ & & & \tilde{\mathfrak w} k i & \\   \ar[r]^{i}_{[j]} & \tilde{\mathfrak w} \ar[r]^{k}_{[j]} & \tilde{\mathfrak w} k \ar[ru]^i_{[k]} \ar[rd]^j_{[k]} &  & \text{and } \quad  \\  & & & \tilde{\mathfrak w} k j  & }   
 \xymatrix{ & & & \tilde{\mathfrak w} jk i \\   \ar[r]^{j}_{[i]} & \tilde{\mathfrak w}j \ar[r]^{k}_{[i]} & \tilde{\mathfrak w}j k \ar[ru]^i_{[k]} \ar[rd]^j_{[k]} &  \\  & & & \tilde{\mathfrak w} j k j }  \]
\end{Lem}

\begin{proof}
First, we prove
\begin{equation} \label{eqn-iio} c_{ik}(\tilde{\mathfrak w} k) \beta_k(\tilde{\mathfrak w} k)  \ge c_{ij}(\tilde{\mathfrak w} k) \beta_j (\tilde{\mathfrak w}k ) .\end{equation}
We compute
\begin{align*}
c_{ik}(\tilde{\mathfrak w} k) \beta_k(\tilde{\mathfrak w} k) &= c_{ik}(\tilde{\mathfrak w})( -\beta_k(\tilde{\mathfrak w}) + c_{jk}(\tilde{\mathfrak w}) \beta_j(\tilde{\mathfrak w}) )= - c_{ik}(\tilde{\mathfrak w}) \beta_k(\tilde{\mathfrak w}) + c_{ik}(\tilde{\mathfrak w}) c_{jk}(\tilde{\mathfrak w}) \beta_j(\tilde{\mathfrak w}), \\
c_{ij}(\tilde{\mathfrak w} k) \beta_j (\tilde{\mathfrak w}j ) &= (-c_{ij}(\tilde{\mathfrak w})+c_{ik}(\tilde{\mathfrak w})c_{jk}(\tilde{\mathfrak w})) \beta_j(\tilde{\mathfrak w}) = -c_{ij}(\tilde{\mathfrak w})\beta_j(\tilde{\mathfrak w})+c_{ik}(\tilde{\mathfrak w})c_{jk}(\tilde{\mathfrak w})\beta_j(\tilde{\mathfrak w}).
\end{align*}
Since we have
\begin{align*}
& c_{ik}(\tilde{\mathfrak w}) \beta_k(\tilde{\mathfrak w}) = c_{ik}(\mathfrak w i (ji)^{n-1}j) \beta_k (\mathfrak w i (ji)^{n-1}j) \\ & \hskip 2 cm \le c_{ij}(\mathfrak w i (ji)^{n-1}j) \beta_j(\mathfrak w i (ji)^{n-1}j)= c_{ij}(\tilde{\mathfrak w})\beta_j(\tilde{\mathfrak w}),
\end{align*}
we see that the inequality \eqref{eqn-iio} holds.

Next, we prove
\begin{equation} \label{eqn-iioo} c_{jk}(\tilde{\mathfrak w} k) \beta_k(\tilde{\mathfrak w} k)  \ge c_{ij}(\tilde{\mathfrak w} k) \beta_i (\tilde{\mathfrak w}k ) .\end{equation}
We compute
\begin{align*}
c_{jk}(\tilde{\mathfrak w} k) \beta_k(\tilde{\mathfrak w} k) &=  - c_{jk}(\tilde{\mathfrak w}) \beta_k(\tilde{\mathfrak w}) + c_{jk}(\tilde{\mathfrak w})^2 \beta_j(\tilde{\mathfrak w}), \\
c_{ij}(\tilde{\mathfrak w} k) \beta_i (\tilde{\mathfrak w}j ) &=  -c_{ij}(\tilde{\mathfrak w})\beta_i(\tilde{\mathfrak w})+c_{ik}(\tilde{\mathfrak w})c_{jk}(\tilde{\mathfrak w})\beta_i(\tilde{\mathfrak w}).
\end{align*}
Since have $c_{jk}(\tilde{\mathfrak w}) \beta_j(\tilde{\mathfrak w}) \ge c_{ik}(\tilde{\mathfrak w})\beta_i(\tilde{\mathfrak w})$ by Lemma \ref{pqa} and $c_{jk}(\tilde{\mathfrak w}) \beta_k(\tilde{\mathfrak w}) \le c_{ij}(\tilde{\mathfrak w})\beta_i(\tilde{\mathfrak w})$ by Lemma \ref{ppa}, the inequality \eqref{eqn-iioo} is proven.

In a similar way, one can prove
\begin{align*}  c_{ik}(\tilde{\mathfrak w}j k) \beta_k(\tilde{\mathfrak w} jk)  &\ge c_{ij}(\tilde{\mathfrak w} jk) \beta_j (\tilde{\mathfrak w}jk ) , \\
c_{jk}(\tilde{\mathfrak w}j k) \beta_k(\tilde{\mathfrak w} jk)  &\ge c_{ij}(\tilde{\mathfrak w} jk) \beta_i (\tilde{\mathfrak w}jk ) ,
\end{align*}
establishing the diagram.
\end{proof}

\begin{Cor} \label{cor-im}
We have
\begin{align*}
\beta_i(\tilde{\mathfrak w} ki) &= s_{\tilde{\mathfrak w}}(s_js_i)^n s_k (s_i s_j)^n (\alpha_i), &
\beta_j(\tilde{\mathfrak w} kj) &= s_{\tilde{\mathfrak w}}(s_js_i)^n s_k (s_i s_j)^n s_i(\alpha_j), \\ 
\beta_i(\tilde{\mathfrak w} jki) &= s_{\tilde{\mathfrak w}}(s_js_i)^{n+1} s_k (s_i s_j)^{n+1} (\alpha_i), &
\beta_j(\tilde{\mathfrak w} jkj) &= s_{\tilde{\mathfrak w}}(s_js_i)^{n+1} s_k (s_i s_j)^{n}s_i (\alpha_j).
\end{align*}
\end{Cor}

Let $\tilde{\mathfrak w}= \mathfrak w i(ji)^n$ for $n \in \mathbb Z_{\ge 0}$ and $\mathfrak v = i(ji)^n$. By \eqref{eqn-ijiki}--\eqref{eqn-ijiki-3}, we have \begin{align*} \phi ( \mathfrak v ki) &= s_i (s_js_i)^{2n} s_k (s_is_j)^n \alpha_i , & \phi ( \mathfrak v kj) &= s_i (s_js_i)^{2n} s_k (s_is_j)^n s_i \alpha_j , \\ \phi ( \mathfrak v jki) &= s_i (s_js_i)^{2n+1} s_k (s_is_j)^{n+1} \alpha_i , & \phi ( \mathfrak v jkj) &= s_i (s_js_i)^{2n+1} s_k (s_is_j)^{n} s_i \alpha_j,\end{align*} 
and obtain
\begin{align*} \beta_i (\tilde{\mathfrak w}k i)  & = s_{\mathfrak w} \phi ({\mathfrak v k} i ) = \psi(\mathfrak w \mathfrak v ki) = \psi(\tilde{\mathfrak w}  ki) , & \beta_j (\tilde{\mathfrak w}k j)  & = s_{\mathfrak w} \phi ({\mathfrak v} k j) =\psi(\tilde{\mathfrak w}  kj), \\ \beta_i (\tilde{\mathfrak w}jk i)  & = s_{\mathfrak w} \phi (\mathfrak v jki )= \psi(\tilde{\mathfrak w} j ki), &  \beta_j (\tilde{\mathfrak w}jkj ) &=  s_{\mathfrak w} \phi ({\mathfrak v} j k j) =\psi(\tilde{\mathfrak w}  jkj) . \end{align*}
Thus we have proven \eqref{eqn-mmm} in this case.

\subsection{The case of $\ell(\mathfrak w)=\delta(\mathfrak w)+3$}

Write $\tilde{\mathfrak u} = \tilde{\mathfrak w}$ and $\tilde{\mathfrak v} =\mathfrak v$, 
or $\tilde{\mathfrak u} =\tilde{\mathfrak w}j$ and $\tilde{\mathfrak v} = \mathfrak v j$, so that $\tilde{\mathfrak u} =\mathfrak w \tilde{\mathfrak v}$.

\begin{Lem} \label{lem-wwab}
We have
\[ \xymatrix{ & & \tilde{\mathfrak u} k ij & \\    \ar[r]^{i} & \tilde{\mathfrak u} ki \ar[ru]^j_{[i]} \ar[rd]^k_{[i]} &  & \text{and } \quad  \\  & & \tilde{\mathfrak u} k ik  & }   
 \xymatrix{ & & \tilde{\mathfrak u} kj i \\   \ar[r]^{j} & \tilde{\mathfrak u}kj \ar[ru]^i_{[j]} \ar[rd]^k_{[j]} &  \\   & & \tilde{\mathfrak u} kj k  }  \]
\end{Lem}

\begin{proof}
First we prove $c_{ij}(\tilde{\mathfrak u}ki) \beta_i (\tilde{\mathfrak u}ki) \ge c_{jk}(\tilde{\mathfrak u}ki) \beta_k (\tilde{\mathfrak u}ki)$. We have
\begin{align*}
c_{ij}(\tilde{\mathfrak u}ki) \beta_i (\tilde{\mathfrak u}ki) &= -c_{ij} (\tilde{\mathfrak u}k) \beta_i (\tilde{\mathfrak u}k) + c_{ij}(\tilde{\mathfrak u}k) c_{ik} (\tilde{\mathfrak u}k) \beta_k (\tilde{\mathfrak u}k), \\
c_{jk}(\tilde{\mathfrak u}ki) \beta_k (\tilde{\mathfrak u}ki) &= - c_{jk} (\tilde{\mathfrak u}k) \beta_k (\tilde{\mathfrak u}k) + c_{ik} (\tilde{\mathfrak u}k) c_{ij}(\tilde{\mathfrak u}k) \beta_k (\tilde{\mathfrak u}k).
\end{align*}
Since $c_{jk} (\tilde{\mathfrak u}k) \beta_k (\tilde{\mathfrak u}k) \ge c_{ij} (\tilde{\mathfrak u}k) \beta_i (\tilde{\mathfrak u}k)$ by Lemma \ref{lem-wwaa}, we are done.

Now we prove $c_{ik}(\tilde{\mathfrak u}ki) \beta_i (\tilde{\mathfrak u}ki) \ge c_{jk}(\tilde{\mathfrak u}ki) \beta_j (\tilde{\mathfrak u}ki)$. The left-hand side is
\[ \text{LHS} = -c_{ik} (\tilde{\mathfrak u}k) \beta_i (\tilde{\mathfrak u}k) + c_{ik}(\tilde{\mathfrak u}k)^2  \beta_k (\tilde{\mathfrak u}k) =  -c_{ik} (\tilde{\mathfrak u}) \beta_i (\tilde{\mathfrak u}) + c_{ik}(\tilde{\mathfrak u}k)^2  \beta_k (\tilde{\mathfrak u}k) , \]
and the right-hand side is
\[ \text{RHS} = -c_{jk} (\tilde{\mathfrak u}k) \beta_j (\tilde{\mathfrak u}k) + c_{ik}(\tilde{\mathfrak u}k) c_{ij} (\tilde{\mathfrak u}k) \beta_j (\tilde{\mathfrak u}k) = -c_{jk} (\tilde{\mathfrak u}) \beta_j (\tilde{\mathfrak u}) + c_{ik}(\tilde{\mathfrak u}k) c_{ij} (\tilde{\mathfrak u}k) \beta_j (\tilde{\mathfrak u}k) .\]

We have $c_{ik}(\tilde{\mathfrak u} k) \beta_k (\tilde{\mathfrak u} k) \ge c_{ij}(\tilde{\mathfrak u} k) \beta_j (\tilde{\mathfrak u} k)$ by Lemma \ref{lem-wwaa}. If $\tilde {\mathfrak u} = \tilde{\mathfrak w}$, then we have $c_{jk}(\tilde{\mathfrak u}) \beta_j (\tilde{\mathfrak u}) \ge c_{ik} (\tilde{\mathfrak u}) \beta_i (\tilde{\mathfrak u})$ by Lemma \ref{pqa} and we have LHS $\ge$ RHS. If $\tilde{\mathfrak u} =\tilde{\mathfrak w} j$, then we compute further and obtain
\begin{align*} \text{LHS} &= -c_{ik} (\tilde{\mathfrak u}) \beta_i (\tilde{\mathfrak u}) + c_{ik}(\tilde{\mathfrak u})^2  ( - \beta_k (\tilde{\mathfrak u}) +c_{ik}(\tilde{\mathfrak u}) \beta_i (\tilde{\mathfrak u}) ) \\ & = c_{ik}(\tilde{\mathfrak u}) ( c_{ik}(\tilde{\mathfrak u})^2 -1) \beta_i (\tilde{\mathfrak u}) - c_{ik}(\tilde{\mathfrak u})^2 \beta_k (\tilde{\mathfrak u}), \\
\text{RHS} &= -c_{jk} (\tilde{\mathfrak u}) \beta_j (\tilde{\mathfrak u}) + c_{ik}(\tilde{\mathfrak u})(- c_{ij} (\tilde{\mathfrak u}) + c_{ik} (\tilde{\mathfrak u}) c_{jk} (\tilde{\mathfrak u}) ) \beta_j (\tilde{\mathfrak u}) \\ & = c_{jk} (\tilde{\mathfrak u}) ( c_{ik}(\tilde{\mathfrak u})^2 -1) \beta_j (\tilde{\mathfrak u}) - c_{ik}(\tilde{\mathfrak u}) c_{ij}(\tilde{\mathfrak u}) \beta_j (\tilde{\mathfrak u}).
\end{align*}
Since we have \[ c_{ij} (\tilde{\mathfrak u}) \beta_j(\tilde{\mathfrak u}) \ge c_{ik} (\tilde{\mathfrak u}) \beta_k(\tilde{\mathfrak u}) \quad \text{ and } \quad c_{ik}(\tilde{\mathfrak u}) \beta_i (\tilde{\mathfrak u}) \ge c_{jk} (\tilde{\mathfrak u}) \beta_j (\tilde{\mathfrak u}) \]
by Lemma \ref{ppa} and \ref{pqa}, we see that LHS $\ge$ RHS.

The inequalities for the second diagram can be proven similarly.
\end{proof}

We need another lemma to complete our proof for this case.
Consider two vectors (or line segments) $\vec{v_1}$ and $\vec{v_2}$, and define $\vec{v_1} * \vec{v_2}$ to be the piecewise linear curve resulting from moving $\vec{v_2}$ to a parallel position to concatenate $\vec{v_1}$ and $\vec{v_2}$ so that the end point of $\vec{v_1}$ and the starting point of $\vec{v_2}$ coincide. Assume that $\vec{v_1} * \vec{v_2}$ starts at $(0,0)$ and ends at a lattice point.
We define 
\[ \upsilon (\vec{v_1} * \vec{v_2}) := p_1 \cdots p_\ell \in \mathfrak W \] which records the consecutive intersections of $\vec{v_1} * \vec{v_2}$ with the sets $\mathcal T_{p_t}$, $t=1, \dots , \ell$ except the starting point and the ending point. This definition is compatible with \eqref{eqn-ups} if we let $\vec{v_2}= \vec{0}$.

In the rest of the paper, we will simply write $\upsilon (\vec{v})$ for $s_{\upsilon (\vec{v})} \in W$ to ease the notation.

\begin{Lem} \label{lem-cru}
Let $\mathfrak vpq \in \mathfrak{W}$. Then we have
\[ \upsilon (\vec{v_p}(\mathfrak vp)) \beta(\vec{v_q}(\mathfrak vp)) = - \beta(\vec{v_q}(\mathfrak v pq)) .\]
 
\end{Lem}

\begin{proof}
Let $\{ p,q,r \} = \{1,2,3 \}$.
Clearly, we have \[ \vec{v_p}(\mathfrak vp)= \tfrac 1 2 \vec{v_p}(\mathfrak vp) + \tfrac 1 2 \vec{v_r}(\mathfrak vp) + \tfrac 1 2 \vec{v_q}(\mathfrak vp) = \tfrac 1 2 \vec{v_q}(\mathfrak vpq) + \tfrac 1 2 \vec{v_q}(\mathfrak vp) . \] The curves $\vec{v_p}(\mathfrak vp)$ and $\tfrac 1 2 \vec{v_q}(\mathfrak vpq) * \tfrac 1 2 \vec{v_q}(\mathfrak vp)$ make a triangle with area $\frac 1 4$. Thus there is no lattice point in the interior of the triangle. Consequently, we have $\upsilon (\vec{v_p}(\mathfrak vp)) = \upsilon \left ( \tfrac 1 2 \vec{v_q}(\mathfrak vpq) * \tfrac 1 2 \vec{v_q}(\mathfrak vp) \right )$.

Since the ending point of $\tfrac 1 2 \vec{v_q}(\mathfrak vpq)$ is in $\mathcal T_q$, we may write
\[ \upsilon \left ( \tfrac 1 2 \vec{v_q}(\mathfrak vpq) * \tfrac 1 2 \vec{v_q}(\mathfrak vp) \right ) = s_{i_1} \cdots s_{i_{k-1}} s_q s_{j_{\ell -1}} s_{j_{\ell -2}} \cdots s_{j_1}, \]
where we have $\beta(\vec{v_q}(\mathfrak v pq)) = s_{i_1} \cdots s_{i_{k-1}} \alpha_q$ and $\beta(\vec{v_q}(\mathfrak v p)) = s_{j_1} \cdots s_{j_{\ell-1}} \alpha_q$. 
Now we have 
\begin{align*} 
\upsilon (\vec{v_p}(\mathfrak vp)) \beta(\vec{v_q}(\mathfrak vp)) &=\upsilon \left ( \tfrac 1 2 \vec{v_q}(\mathfrak vpq) * \tfrac 1 2 \vec{v_q}(\mathfrak vp) \right ) \beta(\vec{v_q}(\mathfrak vp))\\  & = (s_{i_1} \cdots s_{i_{k-1}} s_q s_{j_{\ell -1}} s_{j_{\ell -2}} \cdots s_{j_1} )s_{j_1} \cdots s_{j_{\ell-1}} \alpha_q \\
 & = s_{i_1} \cdots s_{i_{k-1}} s_q \alpha_q = - \beta(\vec{v_q}(\mathfrak v pq)) . 
\end{align*}
\end{proof}

\begin{Cor} \label{cor-abc}
For $(p,q)=(i,j)$ or $(j,i)$, we obtain
\begin{align*}
\beta_q(\tilde{\mathfrak u} kpq) &= s_{\mathfrak w} \phi(\tilde{\mathfrak v} kpq) = \psi (\tilde{\mathfrak u} kpq)  \quad \text{ and } \quad  \beta_k(\tilde{\mathfrak u} kpk) =  s_{\mathfrak w} \phi(\tilde{\mathfrak v} kpk) =  \psi (\tilde{\mathfrak u} kpk) .
\end{align*}
\end{Cor}

\begin{proof}
We first show \begin{equation} \label{eqn-kp}
c_{pq}(\tilde{\mathfrak u}kp)=(\beta_p(\tilde{\mathfrak u}kp), \beta_q(\tilde{\mathfrak u}kp)). \end{equation}
As the other cases are all similar, we only consider the case $\tilde{\mathfrak u}=\tilde{\mathfrak w}$ and $p=j, q=i$. We have
\[c_{pq}(\tilde{\mathfrak u}kp)=c_{ij}(\tilde{\mathfrak w}k)=-c_{ij}(\tilde{\mathfrak w})+c_{ik}(\tilde{\mathfrak w})c_{jk}(\tilde{\mathfrak w}).\] 
On the other hand, since $(\beta_i(\tilde{\mathfrak w}), \beta_k(\tilde{\mathfrak w}))= -c_{ik}(\tilde{\mathfrak w})+c_{ij}(\tilde{\mathfrak w})c_{jk}(\tilde{\mathfrak w})$ by Corollary \ref{cor-var}, we get
\begin{align*}
(\beta_p(\tilde{\mathfrak u}kp), \beta_q(\tilde{\mathfrak u}kp)) &= ( -\beta_j(\tilde{\mathfrak w}k)+c_{jk}(\tilde{\mathfrak w}k)\beta_k(\tilde{\mathfrak w}k) ,\beta_i(\tilde{\mathfrak w})) \\
&=(-\beta_j(\tilde{\mathfrak w})-c_{jk}(\tilde{\mathfrak w})\beta_k(\tilde{\mathfrak w})+c_{jk}(\tilde{\mathfrak w})^2 \beta_j(\tilde{\mathfrak w}), \beta_i(\tilde{\mathfrak w})) \\
&= -c_{ij}(\tilde{\mathfrak w})+c_{jk}(\tilde{\mathfrak w})c_{ik}(\tilde{\mathfrak w})-c_{jk}(\tilde{\mathfrak w})^2 c_{ij}(\tilde{\mathfrak w}) +c_{jk}(\tilde{\mathfrak w})^2 c_{ij}(\tilde{\mathfrak w}) \\&=-c_{ij}(\tilde{\mathfrak w})+c_{jk}(\tilde{\mathfrak w})c_{ik}(\tilde{\mathfrak w}).
\end{align*}
Thus we have proven \eqref{eqn-kp} in this case.

Since we have
\[ \beta_p(\tilde{\mathfrak u} kp)  = s_{\mathfrak w} \beta(\vec{v_p}(\tilde{\mathfrak v}kp)) \quad \text{ and } \quad \beta_q(\tilde{\mathfrak u} kp) = s_{\mathfrak w} \beta(\vec{v_q}(\tilde{\mathfrak v}kp))  ,\]
 we obtain from \eqref{eqn-kp}
\begin{align*} \beta_q(\tilde{\mathfrak u} kpq) &= -\beta_q(\tilde{\mathfrak u}kp) + c_{pq}(\tilde{\mathfrak u}kp) \beta_p(\tilde{\mathfrak u}kp)= - r_{\beta_p(\tilde{\mathfrak u}kp)} \beta_q(\tilde{\mathfrak u}kp) \\ &= - s_{\mathfrak w} \upsilon(\vec{v_p}(\tilde{\mathfrak v}kp)) s_{\mathfrak w}^{-1} s_{\mathfrak w} 
\beta(\vec{v_q}(\tilde{\mathfrak v}kp)) = - s_{\mathfrak w} \upsilon(\vec{v_p}(\tilde{\mathfrak v}kp))  \beta(\vec{v_q}(\tilde{\mathfrak v}kp))
.
\end{align*}
Now it follows from Lemma \ref{lem-cru} that 
\[ \beta_q(\tilde{\mathfrak u}kpq) =  - s_{\mathfrak w} \upsilon(\vec{v_p}(\tilde{\mathfrak v}kp)) \beta(\vec{v_q}(\tilde{\mathfrak v}kp))= 
s_{\mathfrak w} \beta(\vec{v_q}(\tilde{\mathfrak v}kpq))
=s_{\mathfrak w} \phi(\tilde{\mathfrak v}kpq) = \psi (\tilde{\mathfrak u} kpq)  .\]

Similarly, we have 
\begin{equation} \label{eqn-kp-1} c_{pk}(\tilde{\mathfrak u}kp)=(\beta_p(\tilde{\mathfrak u}kp), \beta_k(\tilde{\mathfrak u}kp))  \end{equation}
and compute
\begin{align*}
\beta_k(\tilde{\mathfrak u} kpk) & = - r_{\beta_p(\tilde{\mathfrak u}kp)} \beta_k(\tilde{\mathfrak u}kp) \\ &= - s_{\mathfrak w} \upsilon(\vec{v_p}(\tilde{\mathfrak v}kp)) s_{\mathfrak w}^{-1} s_{\mathfrak w} 
\beta(\vec{v_k}(\tilde{\mathfrak v}kp)) = - s_{\mathfrak w} \upsilon(\vec{v_p}(\tilde{\mathfrak v}kp)) \beta(\vec{v_k}(\tilde{\mathfrak v}kp)) \\ &= 
s_{\mathfrak w} \beta(\vec{v_k}(\tilde{\mathfrak v}kpk)) =   s_{\mathfrak w} \phi(\tilde{\mathfrak v} kpk) =\psi (\tilde{\mathfrak u} kpk).
\end{align*} 
\end{proof}

\subsection{The case of $\ell(\mathfrak w)\geq \delta(\mathfrak w)+4$}

Consider $\hat {\mathfrak w} \in \mathfrak W$ and write
\[ \hat{\mathfrak w} = \tilde{\mathfrak u} k \mathfrak u_1 k \mathfrak u_2 \dots k \mathfrak u_\ell , \]
where we let
\[ \mathfrak u_t = (ij)^{n_t}, (ji)^{n_t}, (ij)^{n_t}i \text{ or } (ji)^{n_t}j \text{ for some } {n_t} \ge 0  \] 
for $t=1,2, \dots, \ell$.

\begin{Lem} \label{lem-fin}
We have
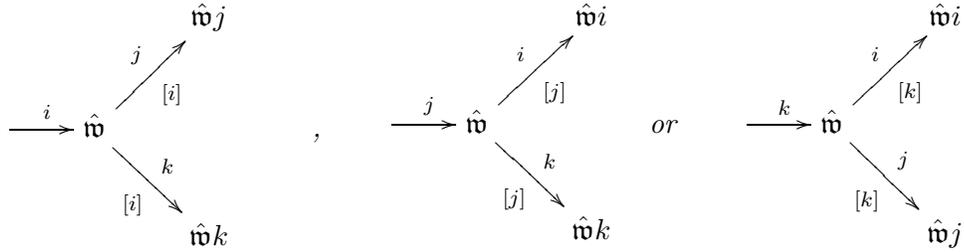
\[ \xymatrix{ &  & \hat{\mathfrak w} j & \\   \ar[r]^{i}_{\phantom{[i]}} & \hat{\mathfrak w} \ar[ru]^j_{[i]} \ar[rd]^k_{[i]} &  & \text{, \quad } \\  &  & \hat{\mathfrak w} k & } \xymatrix{ &  & \hat{\mathfrak w} i & \\   \ar[r]^{j}_{\phantom{[i]}} & \hat{\mathfrak w} \ar[ru]^i_{[j]} \ar[rd]^k_{[j]} &  & \text{\hskip -1 cm  \quad or \quad } \\  &  & \hat{\mathfrak w} k & } \xymatrix{ &  & \hat{\mathfrak w} i \\   \ar[r]^{k}_{\phantom{[i]}} & \hat{\mathfrak w} \ar[ru]^i_{[k]} \ar[rd]^j_{[k]} &  \\  &  & \hat{\mathfrak w}  j} \]
\end{Lem}

\begin{proof}
Since the other cases are similar, we only consider the case $\mathfrak u_\ell =  (ij)^ni$ for $n \ge 0$.
We use induction on $\ell$ and $n$.  When $\ell =1$ and $n=0$, the assertion follows from Lemma \ref{lem-wwab}.
Assume that $\ell \ge 1$ and $n \ge 1$, and suppose that $\mathfrak u_\ell = (ij)^ni$. First we want to prove
\begin{equation} \label{eqn-ioio} c_{ij} (\hat{\mathfrak w}) \beta_i (\hat{\mathfrak w}) \ge c_{jk} (\hat{\mathfrak w}) \beta_k (\hat{\mathfrak w}) .\end{equation}
 Write  $\mathfrak w' = \tilde{\mathfrak u} k \mathfrak u_1 k \mathfrak u_2 \dots \mathfrak u_{\ell -1}$. Let $\gamma = c_{ij}(\mathfrak w' k)$. 
 Then, using the matrix $J_n$ in the proof of Lemma \ref{ppa} with new $\gamma$, we have
\begin{align*}
c_{ij} (\hat{\mathfrak w}) \beta_i (\hat{\mathfrak w}) &= \gamma ( x_n \beta_i (\mathfrak w' ki) - \gamma y_n \beta_j (\mathfrak w' ki)) \\ &= \gamma x_n ( - \beta_i (\mathfrak w' k) + c_{ik} (\mathfrak w' k) \beta_k (\mathfrak w' k)) - \gamma^2 y_n \beta_j (\mathfrak w' k) \\ &= - \gamma x_n \beta_i (\mathfrak w' k) + \gamma x_n c_{ik} (\mathfrak w' k) \beta_k (\mathfrak w' k) - \gamma^2 y_n \beta_j (\mathfrak w' k), \\
c_{jk} (\hat{\mathfrak w}) \beta_k (\hat{\mathfrak w}) &= (x_n c_{jk}(\mathfrak w' ki) - \gamma y_n c_{ik} (\mathfrak w' ki)) \beta_k (\mathfrak w' k) \\ &= x_n ( -c_{jk}(\mathfrak w' k) + \gamma c_{ik}(\mathfrak w' k) ) \beta_k(\mathfrak w' k) - \gamma y_n c_{ik}(\mathfrak w' k) \beta_k(\mathfrak w' k) \\ &=
 - x_n c_{jk} (\mathfrak w' k) \beta_k (\mathfrak w' k)+ \gamma  x_n c_{ik} (\mathfrak w' k)\beta_k (\mathfrak w' k) - \gamma y_n c_{ik}(\mathfrak w' k) \beta_k(\mathfrak w' k) .
\end{align*}
Since we have, by induction,  
\[ c_{ik}(\mathfrak w' k) \beta_k (\mathfrak w' k) \ge \gamma \beta_j (\mathfrak w' k) \text{ and } c_{jk}(\mathfrak w' k) \beta_k (\mathfrak w' k) \ge \gamma \beta_i (\mathfrak w' k), \]
the inequality \eqref{eqn-ioio} follows.

Next we prove
\begin{equation} \label{eqn-ioioio} c_{ik} (\hat{\mathfrak w}) \beta_i (\hat{\mathfrak w}) \ge c_{jk} (\hat{\mathfrak w}) \beta_j (\hat{\mathfrak w}) .\end{equation}
Write $\mathfrak w'' = \tilde{\mathfrak u} k \mathfrak u_1 k \mathfrak u_2 \dots \mathfrak u_{\ell -1}ki(ji)^{n-1}j$.
By induction, we have
\begin{align*}
c_{ik} (\hat{\mathfrak w}) \beta_i (\hat{\mathfrak w}) & = c_{ik} (\mathfrak w'') (- \beta_i (\mathfrak w'')+ c_{ij}(\mathfrak w'')  \beta_j (\mathfrak w'')) = - c_{ik}(\mathfrak w'') \beta_i (\mathfrak w'')+ c_{ik} (\mathfrak w'') c_{ij}(\mathfrak w'') \beta_j (\mathfrak w''), \\ 
c_{jk} (\hat{\mathfrak w}) \beta_j (\hat{\mathfrak w}) & = (-c_{jk}(\mathfrak w'')+c_{ij}(\mathfrak w'') c_{ik}(\mathfrak w'')) \beta_j (\mathfrak w'') =
-c_{jk}(\mathfrak w'') \beta_j (\mathfrak w'') + c_{ik} (\mathfrak w'') c_{ij}(\mathfrak w'') \beta_j (\mathfrak w'').
\end{align*}
Since $c_{jk}(\mathfrak w'') \beta_j (\mathfrak w'') \ge c_{ik} (\mathfrak w'') \beta_i (\mathfrak w'')$ by induction, we see that the inequality \eqref{eqn-ioioio} holds.

\end{proof}

We need another lemma.
\begin{Lem} \label{lem-bibi} 
Assume $\hat{\mathfrak w}$ ends with $q$. Then we have, for $p \neq q$,
\[ c_{pq}(\hat{\mathfrak w}) = (\beta_p(\hat{\mathfrak w}), \beta_q(\hat{\mathfrak w})).\]
\end{Lem}

\begin{proof}
We use induction. If $\hat{\mathfrak w} = \tilde{\mathfrak u}kq$, the assertion follows from \eqref{eqn-kp} and \eqref{eqn-kp-1}. Now assume $\hat{\mathfrak w}= \check{\mathfrak w}rq$ for some $r \neq q$. Then we have
\[  (\beta_p(\hat{\mathfrak w}), \beta_q(\hat{\mathfrak w})) = (\beta_p(\check{\mathfrak w}r), -\beta_q(\check{\mathfrak w}r)+c_{qr}(\check{\mathfrak w}r)\beta_r(\check{\mathfrak w} r)).\]

If $p=r$, then we have by induction
\[  (\beta_p(\hat{\mathfrak w}), \beta_q(\hat{\mathfrak w})) = -c_{pq}(\check{\mathfrak w} r)+2 c_{pq} (\check{\mathfrak w} r) = c_{pq}(\check{\mathfrak w} r) = c_{pq}(\hat{\mathfrak w} ),\]
and we are done.

If $p \neq r$, then we have
\[ c_{pq}(\hat{\mathfrak w}) = c_{pq}(\check{\mathfrak w} r) = -c_{pq}(\check{\mathfrak w}) + c_{pr}(\check{\mathfrak w})c_{qr}(\check{\mathfrak w}), \]
and obtain by induction
\begin{align*}
(\beta_p(\hat{\mathfrak w}), \beta_q(\hat{\mathfrak w})) &= - (\beta_p(\check{\mathfrak w}r), \beta_q(\check{\mathfrak w}r)) +  c_{qr}(\check{\mathfrak w}r)(\beta_p(\check{\mathfrak w}r), \beta_r(\check{\mathfrak w} r)) \\
&=- (\beta_p(\check{\mathfrak w}), \beta_q(\check{\mathfrak w})) +  c_{qr}(\check{\mathfrak w})c_{pr}(\check{\mathfrak w}r) \\ &= -c_{pq}(\check{\mathfrak w}) + c_{qr}(\check{\mathfrak w})c_{pr}(\check{\mathfrak w}).
\end{align*}
This proves the desired identity.
\end{proof}

\begin{Cor}
Assume that $\hat{\mathfrak w} = \tilde{\mathfrak u} k \mathfrak u_1 k \mathfrak u_2 \dots k \mathfrak u_\ell \in \mathfrak W$
where $\mathfrak u_t = (ij)^{n_t}, (ji)^{n_t}, (ij)^{n_t}i$ or $(ji)^{n_t}j$ for some ${n_t} \ge 0$ for $t=1,2, \dots, \ell$. Suppose that $\hat{\mathfrak w}$ does not end with $p$ for $p=i,j$ or $k$.
Then we have
\begin{equation} \beta_p (\hat{\mathfrak w}p) = s_{\mathfrak w} \phi(\tilde{\mathfrak v} k \mathfrak u_1 k \mathfrak u_2 \dots k \mathfrak u_\ell p) = \psi ( \hat{\mathfrak w}p ).\end{equation}
\end{Cor}

\begin{proof}
With Lemma \ref{lem-fin} established, the proof is very similar to that of Corollary \ref{cor-abc}. Suppose that $\hat {\mathfrak w}$ ends with $q$.  By Lemma \ref{lem-bibi} and Lemma \ref{lem-cru}, we have
\begin{align*}
\beta_p(\hat{\mathfrak w} p) & = -\beta_p(\hat{\mathfrak w}) + c_{pq}(\hat{\mathfrak w}) \beta_q(\hat{\mathfrak w})=- r_{\beta_q(\hat{\mathfrak w})} \beta_p(\hat{\mathfrak w}) \\ &= - s_{\mathfrak w} \upsilon (\vec{v_q}( \tilde{\mathfrak v} k \mathfrak u_1 k \mathfrak u_2 \dots k \mathfrak u_\ell )) s_{\mathfrak w}^{-1}s_{\mathfrak w} \beta(\vec{v_p}(\tilde{\mathfrak v} k \mathfrak u_1 k \mathfrak u_2 \dots k \mathfrak u_\ell) )  \\ &= - s_{\mathfrak w} \upsilon (\vec{v_q}( \tilde{\mathfrak v} k \mathfrak u_1 k \mathfrak u_2 \dots k \mathfrak u_\ell ))  \beta(\vec{v_p}(\tilde{\mathfrak v} k \mathfrak u_1 k \mathfrak u_2 \dots k \mathfrak u_\ell) ) \\ &=  s_{\mathfrak w}  \beta(\vec{v_p}(\tilde{\mathfrak v} k \mathfrak u_1 k \mathfrak u_2 \dots k \mathfrak u_\ell p) )  \\ &=   s_{\mathfrak w}  \phi (\tilde{\mathfrak v} k \mathfrak u_1 k \mathfrak u_2 \dots k \mathfrak u_\ell p)
=  \psi ( \hat{\mathfrak w}p).
\end{align*} 

\end{proof}

This completes the proof of Theorem \ref{thm-mm}.

\end{document}